\documentclass[11pt,twoside]{article}
\usepackage{amssymb,mathrsfs,cite,color}
\usepackage{amsfonts,amsmath,latexsym}

\usepackage{indentfirst}
\usepackage[amsmath,thmmarks]{ntheorem}

\numberwithin{equation}{section}

\theoremheaderfont{\normalfont\rmfamily\bf}
\theorembodyfont{\normalfont\it } 
\newtheorem{theorem}{\indent Theorem}[section]
\newtheorem{lemma}{\indent Lemma} [section]

\newtheorem{proposition}{\indent Proposition}[section]
\newtheorem{definition}{\indent Definition} [section]
\newtheorem{remark}{\indent Remark} [section]

\newtheorem{thm}{\indent Theorem}
 
  \theoremstyle{nonumberplain}

\theoremstyle{nonumberplain} \theoremheaderfont{\bf\rmfamily}
\theorembodyfont{\normalfont \rm } \theoremsymbol{$\square$}
\newtheorem{proof}{\indent Proof}

\allowdisplaybreaks

\DeclareMathOperator*{\supp}{supp}
\DeclareMathOperator*{\essinf}{ess\, inf}
\DeclareMathOperator*{\esssup}{ess\, sup}
\def\rr{\mathbb{R}}
\def\rn{{\rr}^n}

\begin{document}

\title{\sf Commutators of multilinear Calder\'on-Zygmund operators
with kernels of Dini's type and applications\footnote{This research
was supported by the National Natural Science Foundation of China
(Grant Nos. 11571160, 11271162 and 11471176).}
  }
\author{Pu Zhang\footnote{Corresponding author.} \  and Jie Sun\\
 {\small\it Department of Mathematics, Mudanjiang Normal
University, Mudanjiang 157011, P. R. China}\\
 {\small\it  E-mail: puzhang@sohu.com {\sc(P. Zhang)},
                sj800816@163.com \sc{(J. Sun)}}
    }
\date{ }
\maketitle

\begin{abstract}

Let $T$ be a multilinear Calder\'on-Zygmund operator of type
$\omega$ with $\omega(t)$ being nondecreasing and satisfying a
kind of Dini's type condition. Let $T_{\Pi\vec{b}}$ be the
iterated commutators of $T$ with $BMO$ functions. The weighted
strong and weak $L(\log{L})$-type endpoint estimates for
$T_{\Pi\vec{b}}$ with multiple weights are established. Some
boundedness properties on weighted variable exponent
Lebesgue spaces are also obtained.

As applications, multiple weighted estimates for iterated
commutators of paraproducts and bilinear pseudo-differential
operators with mild regularity are given.



\end{abstract}


\section{Introduction and Main Results}  \label{introduction}

The study of multilinear Calder\'on-Zygmund theory goes back
to the pioneering works of Coifman and Meyer in 1970s, see e.g.
\cite{cm1,cm2}. This topic was then further investigated by many
authors in the last few decades, see for example \cite{dgy, glmy,
gt1, gt2, h-x, ks, loptt, ls, lz, mn, pptt, pt2014}.

Let $T$ be a multilinear Calder\'on-Zygmund operator with
associated kernel satisfying the standard estimates as in
\cite{gt1} and \cite{loptt}.
For $\vec{b}=(b_1,\cdots,b_m) \in {BMO^m}$, that is $b_j\in
{BMO(\rn)}$ for $j=1, \cdots, m$, the $m$-linear commutator
of $T$ with $\vec{b}$ is defined by
\begin{equation}  \label{equ.1.1} 
T_{\Sigma\vec{b}}(f_1,\cdots,f_m) =\sum_{j=1}^m T^j_{b_j}(\vec{f})
\end{equation}
where
\begin{equation}      \label{equ.j-th}
T^j_{b_j}(\vec{f})=b_jT(f_1,\cdots,f_j,\cdots,f_m)
-T(f_1,\cdots,b_jf_j,\cdots,f_m).
\end{equation}
The iterated commutators of $T$ with $\vec{b}$ is defined by
\begin{equation}  \label{equ.1.2} 
T_{\Pi\vec{b}}(\vec{f})(x)= [b_1,[b_2,...[b_{m-1}, [b_m,T]_m]_{m-1}
... ]_2,]_1(\vec{f})(x).
\end{equation}

For an $m$-linear Calder\'on-Zygmund operator with associated kernel
$K(x,\vec{y})$, the iterated commutator $T_{\Pi\vec{b}}$ can also be
given formally by
$$T_{\Pi\vec{b}}(\vec{f})(x) =\int_{(\rn)^m} \bigg(\prod_{j=1}^m
\big(b_j(x)-b_j(y_j)\big)\bigg) K(x,\vec{y}) f_1(y_1)\cdots
{f_m(y_m)} d\vec{y}.
$$
Here and in what follows, $\vec{y}=(y_1,\cdots,y_m)$,
$(x,\vec{y})=(x,y_1,\cdots,y_m)$ and $d\vec{y}=dy_1\cdots{dy_m}$.

In 2009, Lerner et al. \cite{loptt} developed a multiple weight
theory that adapts to multilinear Calder\'on-Zygmund operators.
They established multiple weighted norm inequalities for
multilinear Calder\'on-Zygmund operators and their commutators
$T_{\Sigma\vec{b}}$.  Recently, P\'erez et al. \cite{pptt} studied
the iterated commutators $T_{\Pi\vec{b}}$ in products of Lebesgue
spaces. Both strong type and weak type estimates with multiple
weights are obtained.

The purpose of this paper is to consider weighted inequalities
with multiple weights for iterated commutators of multilinear
Calder\'on-Zygmund operators of type $\omega$. Some boundedness
properties on weighted variable exponent Lebesgue spaces are also
obtained. In addition, we will give some applications to the
iterated commutators of paraproducts and bilinear
pseudo-differential operators with mild regularity.  This paper
can also be seen as a continuation of our previous one \cite{lz}.

We now recall the definition of multilinear Calder\'on-Zygmund
operators of type $\omega$.

\begin{definition} \label{def.1.1}
Let $\omega(t): [0,\infty)\to [0,\infty)$ be a nondecreasing function.
A locally integrable function $K(x,y_1,\cdots,y_m)$, defined away
from the diagonal $x=y_1=\cdots=y_m$ in $(\rn)^{m+1}$, is called an
$m$-linear Calder\'on-Zygmund kernel of type $\omega$ if, for some
constants $0<\tau<1$, there exists a constant $A>0$ such that
 \begin{equation}  \label{equ.k1} 
  |K(x,y_1,\cdots,y_m)|\le \frac{A}{(|x-y_1|+\cdots +|x-y_m|)^{mn}}
  \end{equation}
for all $(x,y_1,\cdots,y_m) \in (\rn)^{m+1}$ with $x\neq y_j$ for
some $j \in \{1,2, \cdots, m\}$, and
 \begin{equation}  \label{equ.k2} 
 \begin{split}
  &|K(x,y_1,\cdots,y_m)-K(x',y_1,\cdots,y_m)|\\
  &\ \le \frac{A}{(|x-y_1|+\cdots +|x-y_m|)^{mn}}
  \omega\left(\frac{|x-x'|}{|x-y_1|+\cdots +|x-y_m|}\right)
 \end{split}
 \end{equation}
whenever $|x-x'| \le \tau \max_{1\le j\le m}|x-y_j|$, and
 \begin{equation}  \label{equ.k3} 
 \begin{split}
  &|K(x,y_1,\cdots,y_j,\cdots,y_m)-K(x,y_1,\cdots,y_j',\cdots,y_m)|\\
  &\  \le \frac{A}{(|x-y_1|+\cdots +|x-y_m|)^{mn}}
  \omega\left(\frac{|y_j-y_j'|}{|x-y_1|+\cdots +|x-y_m|}\right)
 \end{split}
 \end{equation}
whenever $|y_j-y_j'| \le \tau \max_{1\le i\le m}|x-y_i|$.

We say $T: {\mathscr{S}}(\rn) \times \cdots \times
{\mathscr{S}}(\rn) \to {\mathscr{S'}}(\rn)$ is an $m$-linear
operator with an $m$-linear Calder\'on-Zygmund kernel of type
$\omega$, $K(x,y_1,\cdots,y_m)$, if
$$T(f_1,\cdots,f_m)(x) =\int_{(\rn)^m} K(x,y_1,\cdots,y_m) f_1(y_1)
\cdots {f_m(y_m)} dy_1 \cdots {dy_m}
$$
whenever $x \notin \bigcap_{j=1}^m \supp {f_j}$ and each $f_j\in
{C_c^{\infty}(\rn)}$, $j=1,\cdots, m$.

If $T$ can be extended to a bounded multilinear operator from
$L^{q_1}(\rn)\times \cdots \times L^{q_m}(\rn)$ to $L^{q}(\rn)$
for some $1\le{q_1},\cdots, q_m<\infty$ and $1/q=1/q_1+\cdots +1/q_m$,
then $T$ is called an $m$-linear Calder\'on-Zygmund operator of
type $\omega$, abbreviated to $m$-linear $\omega$-CZO.
\end{definition}

Obviously, when $\omega(t)=t^{\varepsilon}$ for some
$\varepsilon>0$, the $m$-linear $\omega$-CZO is exactly the
multilinear Calder\'on-Zygmund operator studied by Grafakos
and Torres \cite{gt1} and Lerner et al. \cite{loptt}. The
linear Calder\'on-Zygmund operator of type $\omega$ was
studied by Yabuta \cite{y}.

\begin{definition} \label{def.1.2}
Let $\omega(t): [0,\infty) \to [0,\infty)$ be a nondecreasing function.
For $a>0$, we say that $\omega$ satisfies the $Dini(a)$ condition and
wirte $\omega \in Dini(a)$, if
$$|\omega|_{Dini(a)}:=\int_0^1 \frac{\omega^{a}(t)}{t}dt<\infty.
$$
\end{definition}

Obviously, $Dini(a_1)\subset{Dini(a_2)}$ provided $0<a_1<a_2$.

We would like to note that Maldonado and Naibo\cite{mn} studied the
bilinear $\omega$-CZOs when $\omega(t)$ is a nondecreasing, concave
function and belongs to $Dini(1/2)$. Recently, Lu and Zhang \cite{lz}
improve and extend their results essentially by removing the hypothesis
that $\omega$ is concave and reducing the condition $\omega\in{Dini(1/2)}$
to a weaker condition $\omega\in{Dini(1)}$.

\begin{thm}[\cite{lz}]   \label{thm.A} 
Let $\omega\in Dini(1)$ and $T$ be an $m$-linear operator with an
$m$-linear Calder\'on-Zygmund kernel of type $\omega$. Suppose that
for some $1\le q_1, \cdots, q_m \le \infty$ and some $0<q<\infty$
with $1/{q}=1/{q_1}+\cdots + 1/{q_m}$, $T$ maps $L^{q_1}(\rn)\times
\cdots \times L^{q_m}(\rn)$ into $L^{q,\infty}(\rn)$. Then $T$ can
be extended to a bounded operator from $L^1(\rn) \times \cdots
\times L^1(\rn)$ to $L^{1/m,\infty}(\rn)$.
\end{thm}

Denote by $A_{\vec{P}}$ the multiple weight classes introduced by
Lerner et al. \cite{loptt} (see Definition \ref{def.AP} below). The
following multiple weighted strong and weak type estimates for
$m$-linear $\omega$-CZOs were established in \cite{lz}.

\begin{thm}[\cite{lz}]   \label{thm.B} 
Let $T$ be an $m$-linear $\omega$-CZO and $\omega\in {Dini(1)}$.
Let $\vec{P}=(p_1, \cdots, p_m)$ with $1/p=1/p_1 + \cdots +1/p_m$
and ${\vec{w}}\in {A_{\vec{P}}}$.

 {\rm (1)} If $1<p_j<\infty$ for all $j=1, \cdots,
m$, then
$$\big\|T(\vec{f})\big\|_{L^p(\nu_{\vec{w}})}\le {C}
\prod_{j=1}^m\|f_j\|_{L^{p_j}(w_j)}.
$$

{\rm (2)} If $1\le {p_j}<\infty$ for all $j=1, \cdots, m$, and at
least one of the $p_j=1$, then
$$\big\|T(\vec{f})\big\|_{L^{p,\infty}(\nu_{\vec{w}})}\le {C}
\prod_{j=1}^m\|f_j\|_{L^{p_j}(w_j)}.
$$
\end{thm}

Our first result is the following multiple weighted strong-type
estimates for the iterated commutator of
$m$-linear $\omega$-CZO with $BMO$ functions.

\begin{theorem}  \label{thm.1.1}
Let $T$ be an $m$-linear $\omega$-CZO and $\vec{b}\in {BMO^m}$.
Suppose that $\vec{w} \in {A_{\vec{P}}}$ with $1/p=1/p_1 + \cdots
+1/p_m$ and $1<p_j<\infty$, $j=1,\cdots,m$. If $\omega$ satisfies
\begin{equation}  \label{equ.ldini} 
\int_0^1 \frac{\omega(t)}{t}\left(1+\log\frac1{t}\right)^m dt
<\infty,
\end{equation}
then there exists a constant $C>0$ such that
$$\big\|T_{\Pi\vec{b}}(\vec{f})\big\|_{L^p(\nu_{\vec{w}})}\le
{C}\bigg(\prod_{j=1}^m\|b_j\|_{BMO}\bigg)
\prod_{j=1}^m\|f_j\|_{L^{p_j}(w_j)}.
$$
\end{theorem}

\begin{remark} 
Since the commutator has more singularity, the more regular
conditions imposed on the kernel is reasonable. In addition,
although (\ref{equ.ldini}) is stronger than the $Dini(1)$ condition
but it is much weaker than the standard kernel $\omega(t)=t^{\varepsilon}$.
A standard Calder\'on-Zygmund kernel is also a Calder\'on-Zygmund
kernel of type $\omega$ with $\omega(t)$ satisfying
(\ref{equ.ldini}) in the linear case, so does in the multilinear
case.
\end{remark}

It is easy to check that if $\omega$ satisfies (\ref{equ.ldini}),
then $\omega \in {Dini(1)}$ and
\begin{equation}   \label{equ.ldini-e}  
\sum_{k=1}^{\infty}k^m \cdot\omega(2^{-k}) \approx \int_0^1
\frac{\omega(t)}{t}\left(1+\log\frac1{t}\right)^m dt <\infty.
\end{equation}

For the multiple weighted weak-type estimate, we have the following
result.

\begin{theorem} \label{thm.1.2}
Let $T$ be an $m$-linear $\omega$-CZO and $\vec{b}\in {BMO^m}$.
If $\vec{w} \in {A_{(1,\cdots,1)}}$ and $\omega$ satisfies
(\ref{equ.ldini}), then there is a constant $C>0$ depending on
$\|{b_j}\|_{BMO}$, $j=1,\cdots,m$, such that for any $\lambda>0$,
$$\nu_{\vec{w}}\big(\big\{ x\in \rn: \big|T_{\Pi\vec{b}}(\vec{f})(x)
\big|>\lambda^m \big\}\big) \le {C} \prod_{j=1}^m \left(
\int_{\rn}\Phi^{(m)} \bigg(\frac{|f_j(x)|}{\lambda}\bigg)
w_j(x)dx\right)^{1/m},
$$
here and in the sequel, $\Phi(t)=t(1+\log^+t)$ and
$\Phi^{(m)}=\underbrace{\Phi \circ \cdots \circ \Phi}_m$.
\end{theorem}

\begin{remark}
P\'erez et al. \cite{pptt} proved the same results as Theorems \ref{thm.1.1}
and \ref{thm.1.2} when $\omega(t)=t^{\varepsilon}$ for some
$\varepsilon>0$. So, Theorems \ref{thm.1.1} and \ref{thm.1.2}
improve the main results in \cite{pptt} essentially. We also note
that similar results for $T_{\Sigma\vec{b}}$ were proved in
\cite{lz}. For the linear commutator of Calder\'on-Zygmund operator,
see \cite{ll} and \cite{zx}.
\end{remark}

Next, we will study the boundedness of iterated commutators of
$m$-linear $\omega$-CZOs on weighted variable exponent Lebesgue
spaces. We now recall some definition and notation. For more
information on function spaces with variable exponent,
we refer to \cite{cu-f} and \cite{dhhr}.

A measurable function $p(\cdot): \rn \to (0,\infty)$ is called
a variable exponent. For any variable exponent $p(\cdot)$, let
$$ p^{-}:=\essinf_{x\in \rn}p(x) ~~\mathrm{and}~~
{p^{+}:}=\esssup_{x\in \rn}p(x).
$$
Denote by $\mathscr{P}_0$ the set of all variable
exponents $p(\cdot)$ satisfying $0<p^{-} \le {p^{+}} <\infty$
and by $\mathscr{P}$ the collection of all variable exponents
$p(\cdot): \rn\to[1,\infty)$ with $1<p^{-}\le {p^{+}}<\infty$.

Given a variable exponent $p(\cdot) \in \mathscr{P}_0$, the
variable exponent Lebesgue space is defined by
$$ L^{p(\cdot)}(\rn)=\bigg\{f~ \mbox{measurable}:
\int_{\rn} \bigg(\frac{|f(x)|}{\lambda}\bigg)^{p(x)} dx <\infty
~\mbox{for some constant}~ \lambda>0\bigg\}.
$$
The set $L^{p(\cdot)}(\rn)$ is a quasi-Banach space (Banach
space if $p(\cdot) \in \mathscr{P}$) with the quasi-norm
(norm if $p(\cdot) \in \mathscr{P}$) given by
$$\|f\|_{L^{p(\cdot)}(\rn)}=\inf \bigg\{ \lambda>0: \int_{\rn}
\bigg(\frac{|f(x)|}{\lambda} \bigg)^{p(x)} dx \le 1 \bigg\}.
$$

For $p(\cdot)\in \mathscr{P}$, we define the conjugate
exponent $p'(\cdot)$ by $1/p(\cdot) +1/p'(\cdot)=1$
with $1/\infty=0$.

\begin{definition}[\cite{cfn}] \label{def.1.3}
A locally integrable function $v$ with $0<v<\infty$ almost
everywhere is called a weight. Given a variable exponent
$p(\cdot)\in \mathscr{P}$, we say that $v \in{A_{p(\cdot)}}$
if
$$[v]_{A(\cdot)} =\sup_{B} |B|^{-1}\|v\chi_B\|_{L^{p(\cdot)}(\rn)}
\|v^{-1}\chi_B\|_{L^{p'(\cdot)}(\rn)} < \infty,
$$
where the supremum is taken over all balls $B\subset \rn$.
\end{definition}

The weighted variable exponent Lebesgue space
$L^{p(\cdot)}_{v}(\rn)$ is defined to be the set of all
measurable functions $f$ such that $fv\in {L^{p(\cdot)}(\rn)}$,
and we define
$\|f\|_{L_{v}^{p(\cdot)}(\rn)}=\|fv\|_{L^{p(\cdot)}(\rn)}$.

We say that a variable exponent $p(\cdot)$ satisfies
the globally log-H\"older continuous condition,
if there exist positive constants $C_0, C_{\infty}$
and $p_{\infty}$ such that
\begin{equation}    \label{equ.log1}
|p(x)-p(y) \le \frac{C_0}{-\log(|x-y|)}, \quad x,y \in \rn,
 \ |x-y|\le 1/2
\end{equation}
and
\begin{equation}    \label{equ.log2}
|p(x)-p_{\infty}| \le \frac{C_{\infty}}{\log(e+|x|)}, \quad x \in \rn.
\end{equation}

For iterated commutator of $m$-linear $\omega$-CZO, we have the
following results.

\begin{theorem} \label{thm.1.3}   
Let $T$ be an $m$-linear $\omega$-CZO with $\omega$ satisfying
(\ref{equ.ldini}) and $\vec{b}\in {BMO^m}$. Suppose that
$p(\cdot) \in \mathscr{P}_0$ and $p_1(\cdot), \cdots, p_m(\cdot)
\in \mathscr{P}$ so that
${1}/{p(\cdot)} ={1}/{p_1(\cdot)}+\cdots +{1}/{p_m(\cdot)}.$
If $p(\cdot), p_1(\cdot), \cdots, p_m(\cdot)$ satisfy (\ref{equ.log1})
and (\ref{equ.log2}), $v_j \in {A_{p_j(\cdot)}}$,
$j=1,\cdots,m$, and $v=\prod_{j=1}^m v_j$, then $T_{\Pi\vec{b}}$
can be extended to a bounded operator from
$L_{v_1}^{p_1(\cdot)}(\rn)\times \cdots \times {L_{v_m}^{p_m(\cdot)}(\rn)}$
to $L_v^{p(\cdot)}(\rn)$, that is, there exists a positive constant
$C$ such that
$$\big\|T_{\Pi\vec{b}}(\vec{f})\big\|_{L_v^{p(\cdot)}(\rn)}\le
{C}\prod_{j=1}^{m}\|f_j\|_{L_{v_j}^{p_j(\cdot)}(\rn)}.
$$
\end{theorem}

\begin{remark}
By the same procedure as the proof of Theorem \ref{thm.1.3}, we can obtain
similar estimates in weighted variable exponent Lebesgue spaces for
$m$-linear $\omega$-CZO and $m$-linear commutator $T_{\Sigma{\vec{b}}}$. We
leave the details to the readers.
\end{remark}

This paper is arranged as follows. In Section \ref{pre}, we recall
some basic definitions and known results. Section \ref{proof} is
devoted to proving Theorems \ref{thm.1.1} and \ref{thm.1.2}. We give
the proof of Theorem \ref{thm.1.3} in Section \ref{variable}. In the
last section, we apply our results to paraproducts and bilinear
pseudo-differential operators with mild regularity.


\section{Definitions and Preliminaries}   \label{pre}

\subsection{Orlicz norms}

For $\Phi=t(1+\log^{+}t)$, the $\Phi$-average of a function $f$ on a
cube $Q\subset \rn$ is defined by
$$\|f\|_{L(\log{L}),Q}\equiv \|f\|_{\Phi,Q}:=\inf\left\{\lambda>0:
\frac1{|Q|} \int_{Q}\Phi\left(\frac{|f(x)|}{\lambda}\right)dx\le
1\right\}.
$$
The maximal function associated to $\Phi(t)=t(1+\log^+t)$ is defined
by
$$M_{L(\log{L})}(f)(x)=\sup_{Q\ni {x}}\|f\|_{L(\log{L}),Q},
$$
where the supremum is taken over all the cubes containing $x$.

It is not difficult to check the following pointwise equivalence
(see (21) in \cite{pe})
$$M_{L(\log{L})}(f)(x) \approx M^2f(x),
$$
where $M$ is the Hardy-Littlewood maximal function and $M^2=M\circ {M}$.

For a cube $Q$ and a locally integrable function $f$, we
denote by $f_Q=(f)_Q=\frac1{|Q|}\int_Qf(y)dy$.
Let $b \in {BMO(\rn)}$, by the generalized H\"older's
inequality in Orlicz spaces (see \cite[page 58]{rr}) and
John-Nirenberg's inequality, we have (see also (2.14) in
\cite{loptt})
\begin{equation}       \label{equ.bmo1}  
\frac1{|Q|}\int_Q |b(x)-b_Q||f(x)|dx \le
{C}\|b\|_{BMO}\|f\|_{L(\log{L}),Q}.
\end{equation}

Let $t>1$ and $Q$ be a cube in $\rn$. Denote by $tQ$ the cube that is
concentric with $Q$ and whose side length is $t$ times  the side
length of $Q$. Then, there exists a dimensional constant $C_n$ such
that for any $b\in{BMO(\rn)}$, we have (see \cite[page 166]{gra})
\begin{equation}       \label{equ.bmo2}  
|b_Q-b_{tQ}|\le {C_n}\log(t+1)\|b\|_{BMO}.
\end{equation}

Moreover, for any $r>0$ and $t>1$,
\begin{equation}       \label{equ.bmo3}  
\begin{split}
\bigg(\frac1{|Q|}\int_Q|b(z)-b_{tQ}|^{r}dz\bigg)^{1/r}
&\le{C}\bigg(\frac1{|Q|}\int_Q\big|b(z)-b_{Q}\big|^{r}dz
  +|b_Q-b_{tQ}|^r\bigg)^{1/r} \\
&\le{C}(1+\log{t})\|b\|_{BMO}.
\end{split}
\end{equation}


\subsection{Sharp maximal function and $A_p$ weights }

The sharp maximal function of
Fefferman and Stein \cite{fs} is defined by
$$M^{\sharp}(f)(x)=\sup_{Q\ni{x}}\inf_c\frac1{|Q|}\int_Q|f(y)-c|dy
\approx \sup_{Q\ni{x}}\frac1{|Q|}\int_Q|f(y)-f_Q|dy.
$$

Let $M$ be the usual Hardy-Littlewood maximal function, for
$0<\delta<\infty$, we define the maximal functions $M_{\delta}$ and
$M_{\delta}^{\sharp}$ by
$$M_{\delta}(f)=\big[M(|f|^{\delta})\big]^{1/\delta} \quad
{\mbox{and}} \quad M_{\delta}^{\sharp}(f)=\big[M^{\sharp}
\big(|f|^{\delta})\big]^{1/\delta}.
$$

Let $w$ be a nonnegative locally integrable function and
$1<p<\infty$. We say $w \in {A_p}$ if there is a
constant $C>0$ such that for any cube $Q$,
$$\bigg(\frac1{|Q|}\int_Q w(x)dx\bigg) \bigg(\frac1{|Q|}\int_Q
w(x)^{1-p'}dx\bigg)^{p-1} \le C.
$$

We say $w \in {A_1}$ if there is a constant $C>0$ such that
$Mw(x)\le {C}w(x)$ almost everywhere. And we define
$A_{\infty}=\bigcup_{p\ge 1}A_p$. See \cite{gar-rub} or \cite{s} for
more information.

The following relationships between $M_{\delta}^{\sharp}$ and
$M_{\delta}$ to be used is a version of the classical ones due to
Fefferman and Stein \cite{fs}, see also \cite[page 1228]{loptt}.

\begin{lemma}   \label{lem.fs}  
{\rm (1)} Let $0<p, \delta <\infty$ and $w \in {A_{\infty}}$. Then
there exists a constant $C>0$ (depending on the $A_{\infty}$
constant of $w$) such that
$$\int_{\rn}\big[M_{\delta}(f)(x)\big]^pw(x)dx \le
{C}\int_{\rn}\big[M_{\delta}^{\sharp}(f)(x)\big]^pw(x)dx,
$$
for every function $f$ such that the left-hand side is finite.

{\rm (2)}  Let $0<\delta <\infty$ and $w \in {A_{\infty}}$. If
$\varphi : (0,\infty) \to (0,\infty)$ is doubling, then there exists
a constant $C>0$ (depending on the $A_{\infty}$ constant of $w$ and
the doubling condition of $\varphi$) such that
$$\sup_{\lambda>0}\varphi(\lambda)w\big(\big\{y\in\rn:
M_{\delta}(f)(y)>\lambda\big\}\big) \le {C}\sup_{\lambda>0}
\varphi(\lambda) w\big(\big\{y\in\rn: M_{\delta}^{\sharp}(f)(y)
>\lambda\big\}\big),
$$
for every function $f$ such that the left-hand side is finite.
\end{lemma}


\subsection{Multilinear maximal functions and multiple weights}

The following multilinear maximal functions that adapts to the
multilinear Calder\'on-Zygmund theory are introduced by Lerner et al.
in \cite{loptt}.

\begin{definition}     \label{def.m-linear M}  
For all $m$-tuples $\vec{f}=(f_1, \cdots, f_m)$ of locally
integrable functions and $x\in\rn$, the multilinear maximal
functions ${\cal{M}}$ and ${\cal{M}}_r$ are defined by
$${\cal{M}}(\vec{f})(x)=\sup_{Q\ni x}\prod_{j=1}^m \frac1{|Q|}
\int_{Q}|f_j(y_j)|dy_j,
$$
and
$${\cal{M}}_r(\vec{f})(x)=\sup_{Q\ni x}\prod_{j=1}^m \bigg(
\frac1{|Q|}\int_{Q}|f_j(y_j)|^rdy_j\bigg)^{1/r},~~ for ~ r>1.
$$
The maximal functions related to function
$\Phi(t)=t(1+\log^+t)$ are defined by
$${\cal{M}}^i_{L(\log{L})}(\vec{f})(x)=\sup_{Q\ni x}
\|f_i\|_{L(\log{L}),Q} \prod_{j=1(j\neq{i})}^m
\frac1{|Q|}\int_{Q}|f_j(y_j)|dy_j
$$
and
$${\cal{M}}_{L(\log{L})}(\vec{f})(x)=\sup_{Q\ni {x}}
\prod_{j=1}^m\|f_j\|_{L(\log{L}),Q},
$$
where the supremum is taken over all the cubes $Q$ containing $x$.
\end{definition}

Obviously, if $r>1$ then the following pointwise estimates hold (see
(4.3) in \cite{lz})
\begin{equation}       \label{equ.mmm}   
{\cal{M}}(\vec{f})(x) \le {C}{\cal{M}}^i_{L(\log{L})}(\vec{f})(x)
\le {C'}{\cal{M}}_{L(\log{L})}(\vec{f})(x) \le {C''}
{\cal{M}}_r(\vec{f})(x).
\end{equation}

The following multiple ${A_{\vec{P}}}$ conditions were introduced
by Lerner et al. \cite{loptt}.

\begin{definition}[\cite{loptt}]  \label{def.AP} 
Let $\vec{P}=(p_1, \cdots, p_m)$ and $1/p = 1/p_1 +\cdots +1/p_m$
with $1\le p_1, \cdots, p_m<\infty$. Given $\vec{w}=(w_1, \cdots,
w_m)$ with each $w_j$ being nonnegative measurable, set
$$\nu_{\vec{w}}=\prod_{j=1}^m w_j^{p/p_j}.
$$
We say that $\vec{w}$ satisfies the ${A_{\vec{P}}}$ condition and
write $\vec{w}\in {A_{\vec{P}}}$, if
$$\sup_Q \left(\frac{1}{|Q|}\int_Q\nu_{\vec{w}}(x)dx\right)^{1/p}
\prod_{j=1}^m \left(\frac{1}{|Q|}\int_{Q}
w_j(x)^{1-p'_j}dx\right)^{1/p'_j} <\infty,
$$
where the supremum is taken over all cubes $Q\subset \rn$, and the
term $\big(\frac{1}{|Q|}\int_Q w_j(x)^{1-p'_j}dx\big)^{1/p'_j}$ is
understood as $\left(\inf_Q w_j\right)^{-1}$ when $p_j=1$.
\end{definition}

In \cite{loptt}, the characterizations of the multiple weight
classes $A_{\vec{P}}$ in terms of the multilinear maximal
function ${\cal{M}}$ are proved in Theorem 3.3 and Theorem 3.7.
We restate it as follows.

\begin{lemma}[\cite{loptt}]  \label{lem.wm}
Let $\vec{P}=(p_1, \cdots, p_m)$, $1\le p_1, \cdots, p_m<\infty$ and
$1/p = 1/p_1 +\cdots +1/p_m$.

{\rm (1)} If $1<p_1, \cdots, p_m<\infty$, then $\cal{M}$ is bounded
from $L^{p_1}(w_1) \times \cdots \times {L^{p_m}(w_m)}$ to
$L^p(\nu_{\vec{w}})$ if and only if ${\vec{w}}=(w_1, \cdots, w_m)
\in{A_{\vec{P}}}$.

{\rm (2)} If $1\le p_1, \cdots, p_m<\infty$, then $\cal{M}$ is
bounded from $L^{p_1}(w_1) \times \cdots \times {L^{p_m}(w_m)}$ to
$L^{p,\infty}(\nu_{\vec{w}})$ if and only if ${\vec{w}}=(w_1,
\cdots, w_m) \in{A_{\vec{P}}}$.
\end{lemma}

The characterization of the multiple weight classes $A_{\vec{P}}$
in terms of the Muckenhoupt weights, which will be used later, is
established by Lerner et. al. \cite[Theorem 3.6]{loptt}.

\begin{lemma}[\cite{loptt}] \label{lem.ww}
Let ${\vec{w}}=(w_1,\cdots,w_m)$, ${\vec{P}}=({p_1},\cdots,p_m)$,
$1\le {p_1},\cdots,p_m < \infty$ and $1/p=1/{p_1}+\cdots + 1/p_m$.
Then ${\vec{w}}\in {A_{\vec{P}}}$ if and only if
\begin{displaymath}
\left\{\begin{array}{ll}
w_j^{1-p'_j} \in {A_{mp'_j}}, &~~ j=1, \cdots, m, \\
\nu_{\vec{w}}  \in {A_{mp}} ,
\end{array}\right.
\end{displaymath}
where the condition $w_j^{1-p'_j} \in {A_{mp'_j}}$ in the case
$p_j=1$ is understood as $w_j^{1/m} \in {A_1}$.
\end{lemma}

The following boundedness of ${\cal{M}}_r$ is contained in the proof
of Theorem 3.18 of \cite{loptt}.

\begin{lemma}   \label{lem.mrw}
Let $\vec{P}=(p_1,\cdots,p_m)$, $1<p_1,\cdots,p_m<\infty$ and
$1/p=1/p_1+\cdots+1/p_m$. If ${\vec{w}} \in {A_{\vec{P}}}$, then
there exists a constant $r>1$ such that ${\cal{M}}_r$ is bounded
from $L^{p_1}(w_1) \times \cdots \times L^{p_m}(w_m)$ to
$L^p(\nu_{\vec{w}})$.
\end{lemma}

Moreover, the maximal function ${\cal{M}}_{L(\log{L})}$ satisfies
the following weak-type estimates obtained in \cite[Theorem 4.1]{pptt},
which will be used later.

\begin{lemma}[\cite{pptt}]   \label{lem.m-llogl}
Let ${\vec{w}} \in {A_{(1,\cdots,1)}}$. Then there exists a constant
$C>0$ such that
$$\nu_{\vec{w}}\big(\big\{ x\in \rn:
{\cal{M}}_{L(\log{L})}(\vec{f})(x)>t^m\big\}\big) \le {C}
\prod_{j=1}^m \left(\int_{\rn}
\Phi^{(m)}\bigg(\frac{|f_j(x)|}{t}\bigg)w_j(x)dx \right)^{1/m}.
$$
\end{lemma}

\subsection{Kolmogorov's inequality}

We will also need the following Kolmogorov's inequality (see
\cite[page 485]{gar-rub} or \cite[(2.16)]{loptt}).

\begin{lemma}   \label{lem.kolm}
Let $0<p<q<\infty$, then there is a positive constant $C=C_{p,q}$
such that for any measurable function $f$,
$$|Q|^{-1/p}\|f\|_{L^p(Q)}\le {C}|Q|^{-1/q}\|f\|_{L^{q,\infty}(Q)}.
$$
\end{lemma}

\section{Proof of Theorems \ref{thm.1.1} and \ref{thm.1.2}}  \label{proof}

Before proving the results, we first remark that the specific
value of $\tau \in (0,1)$ in Definition \ref{def.1.1}
is immaterial. From now on, we may assume the
constant $\tau$ appearing in Definition \ref{def.1.1} equals to $1/2$ for
simplicity. The argument with trivial modifications is also applicable to
any specific value of $\tau \in (0,1)$.

We now introduce some notation for convenience. For positive
integers $m$ and $j$ with $1\le{j}\le
{m}$, we denote by $C_j^m$ the family of all finite subsets $\sigma
= \{\sigma(1), \cdots, \sigma(j)\}$ of $\{1,\cdots,m\}$ of $j$
different elements, where we always take $\sigma(k)<\sigma(j)$ if
$k<j$. For any $\sigma\in {C_j^m}$, we write
$\sigma'=\{1,\cdots,m\}\setminus \sigma$ and $C_0^m=\emptyset$.

Given an $m$-tuple $\vec{b}=(b_1,\cdots,b_m) \in{BMO^m}$ and $\sigma
= \{\sigma(1), \cdots, \sigma(j)\} \in {C_j^m}$ with $\sigma'
=\{\sigma'(1), \cdots, \sigma'(m-j)\}$, we denote by
$\vec{b}_{\sigma}=(b_{\sigma(1)}, \cdots, b_{\sigma(j)})$ and
$\vec{b}_{\sigma'}=(b_{\sigma'(1)}, \cdots, b_{\sigma'(m-j)})$.

For $\sigma \in {C_j^m}$ and $\vec{b}_{\sigma}=(b_{\sigma(1)},
\cdots, b_{\sigma(j)}) \in {BMO^j}$, similar to (\ref{equ.1.2}), we
can define the following iterated commutator
$$T_{\Pi\vec{b}_\sigma}(\vec{f})(x)= [b_{\sigma(1)},[b_{\sigma(2)},
... [b_{\sigma(j-1)}, [b_{\sigma(j)},T]_{\sigma(j)}]_{\sigma(j-1)}
...]_{\sigma(2)},]_{\sigma(1)}(\vec{f})(x).
$$
This is, formally
$$T_{\Pi\vec{b}_{\sigma}}(\vec{f})(x) =\int_{(\rn)^m}
\bigg(\prod_{i=1}^j \big(b_{\sigma(i)}(x)-b_{\sigma(i)}(y_{\sigma(i)})\big)
\bigg) K(x,\vec{y}) f_1(y_1)\cdots {f_m(y_m)} d\vec{y}.
$$

Obviously, $T_{\Pi\vec{b}_{\sigma}} =T_{\Pi\vec{b}}$ if $\sigma=
\{1,\cdots,m\}$ and $T_{\Pi\vec{b}_{\sigma}} =T_{b_j}^j$ if
$\sigma=\{j\}$.

To prove Theorems \ref{thm.1.1} and \ref{thm.1.2}, we first give the
following pointwise estimates.

\begin{proposition}  \label{prop.sharp} 
Let $T$ be an $m$-linear $\omega$-CZO with $\omega$ satisfying
(\ref{equ.ldini}) and $\vec{b} \in {BMO^m}$. Let $0<\delta
<\varepsilon$ and $0<\delta<1/m$. Then, there exists a positive
constant $C$, depending on $\delta$ and $\varepsilon$, such that
\begin{equation}   \label{equ.sharp0}
\begin{split}
M_{\delta}^{\sharp}\big(T_{\Pi\vec{b}}(\vec{f})\big)(x) &\le{C}
\bigg(\prod_{j=1}^m \|b_j\|_{BMO}\bigg)
\Big\{{\mathcal{M}_{L(\log{L})}}(\vec{f})(x)
+M_{\varepsilon}\big(T(\vec{f})\big)(x) \Big\}\\
&\quad +C\sum_{j=1}^{m-1} \sum_{\sigma \in {C_j^m}}
\bigg(\prod_{i=1}^j \|b_{\sigma(i)}\|_{BMO}\bigg)
M_{\varepsilon}\big(T_{\Pi\vec{b}_{\sigma'}}(\vec{f})\big)(x),
\end{split}
\end{equation}
for all bounded measurable functions $f_1, \cdots, f_m$ with compact support.
\end{proposition}

\begin{proof}
For the sake of clarity and simplicity, we prove only the case
$m=2$. The general case can be obtained by a similar argument 
with only trivial modifications.

Fix $b_1, b_2 \in {BMO(\rn)}$, for any constants $\lambda_1$ 
and $\lambda_2$, since
\begin{equation}   \label{equ.sharp-T}
T_{\Pi\vec{b}}(f_1,f_2) = b_1b_2T(f_1,f_2)-b_2T(b_1f_1,f_2) 
 -b_1T(f_1,b_2f_2)+T(b_1f_1,b_2f_2),
\end{equation}
we can rewrite it as follows
\begin{align*}
T_{\Pi\vec{b}}(f_1,f_2)
&= (b_1-\lambda_1)(b_2-\lambda_2)T(f_1,f_2)
 -(b_2-\lambda_2)T\big((b_1-\lambda_1)f_1,f_2\big)\\
  &\quad -(b_1-\lambda_1) T\big(f_1,(b_2-\lambda_2)f_2\big)
   +T\big((b_1-\lambda_1)f_1,(b_2-\lambda_2)f_2\big)\\
&= - (b_1-\lambda_1)(b_2-\lambda_2)T(f_1,f_2)
 +(b_2-\lambda_2)T_{b_1-\lambda_1}^1(f_1,f_2)\\
   &\quad +(b_1(x)-\lambda_1) T_{b_2-\lambda_2}^2(f_1,f_2)
    +T\big((b_1-\lambda_1)f_1,(b_2-\lambda_2)f_2\big),
\end{align*}
where
$$T_{b_1-\lambda_1}^1(f_1,f_2)(x) =(b_1(x)-\lambda_1)T(f_1,f_2)(x)
  -T\big((b_1-\lambda_1)f_1,f_2\big)(x)
$$
and
$$T_{b_2-\lambda_2}^2(f_1,f_2)(x) =(b_2(x)-\lambda_2)T(f_1,f_2)(x)
  -T\big(f_1,(b_2-\lambda_2)f_2\big)(x).
$$

Now we fix $x\in\rn$. For any cube $Q$ centered at $x$, set
$Q^{*}=8\sqrt{n}Q$ and let $\lambda_j=(b_j)_{Q^{*}}$ be the average
of $b_j$ on $Q^{*}$, $j=1,2$. Since $0<\delta<1/2$, then, for any
real number $c$,
\begin{equation}    \label{equ.sharp1}
\begin{split}
\bigg(\frac1{|Q|} \int_Q\Big| |T_{\Pi\vec{b}}(f_1,f_2)(z)|^{\delta}
 -|c|^{\delta}\Big|dz\bigg)^{1/{\delta}} &\le\bigg(\frac1{|Q|}
  \int_Q\big| T_{\Pi\vec{b}}(f_1,f_2)(z)
   -c\big|^{\delta} dz\bigg)^{1/{\delta}}\\
&\le {C}(I_1 +I_2 +I_3 +I_4),
\end{split}
\end{equation}
where
\begin{equation*}
\begin{split}
&I_1=\bigg(\frac1{|Q|}\int_Q \big| (b_1(z)-\lambda_1)
 (b_2(z)-\lambda_2)T(f_1,f_2)(z)\big|^{\delta}dz\bigg)^{1/{\delta}},\\
&I_2=\bigg(\frac1{|Q|}\int_Q \big|(b_2(z)-\lambda_2)
 T_{b_1-\lambda_1}^1(f_1,f_2)(z)\big|^{\delta}dz\bigg)^{1/{\delta}},\\
&I_3=\bigg(\frac1{|Q|}\int_Q \big| (b_1(z)-\lambda_1)
 T_{b_2-\lambda_2}^2(f_1,f_2)(z)\big|^{\delta}dz\bigg)^{1/{\delta}},\\
&I_4=\bigg(\frac1{|Q|}\int_Q \big| T\big((b_1-\lambda_1)f_1,
 (b_2-\lambda_2)f_2\big)(z)-c\big|^{\delta} dz\bigg)^{1/{\delta}}.
\end{split}
\end{equation*}

For any $1<r_1,r_2,r_3<\infty$ with $1/r_1+1/r_2+1/r_3=1$ and
$r_3<{\varepsilon}/{\delta}$, it follows from H\"older's inequality
and (\ref{equ.bmo3}) that
\begin{equation}    \label{equ.sharp2}
\begin{split}
I_1 &\le \bigg(\frac1{|Q|}\int_Q\big|b_1(z)-(b_1)_{Q^*}
 \big|^{r_1\delta} dz\bigg)^{\frac{1}{r_1\delta}}
  \bigg(\frac1{|Q|}\int_Q\big|b_2(z)-(b_2)_{Q^*}\big|^{r_2\delta}
dz\bigg)^{\frac1{r_2\delta}} \\
& \qquad \times \bigg(\frac1{|Q|}\int_Q\big|T(f_1,f_2)(z)
  \big|^{r_3\delta} dz\bigg)^{\frac1{r_3\delta}}\\
&\le{C}\|b_1\|_{BMO}\|b_2\|_{BMO}
  M_{r_3\delta}\big(T(f_1,f_2))(x)\\
&\le{C}\|b_1\|_{BMO}\|b_2\|_{BMO}M_{\varepsilon}\big(T(f_1,f_2))(x).
\end{split}
\end{equation}

For the second term $I_2$, let $1<t_1,t_2<\infty$ with $1=1/t_1
+1/t_2$ and $t_2<{\varepsilon}/{\delta}$. Then H\"older's
inequality together with (\ref{equ.bmo3}) gives
\begin{equation}     \label{equ.sharp3}
\begin{split}
I_2&\le \bigg(\frac1{|Q|} \int_Q|b_2(z)-(b_2)_{Q^{*}}|^{t_1\delta}
 dz\bigg)^{\frac1{t_1\delta}}
  \bigg(\frac1{|Q|}\int_Q\big|T^1_{b_1-\lambda_1}(f_1,f_2)(z)
   \big|^{t_2\delta}dz\bigg)^{\frac11{t_2\delta}}\\
&\le {C} \|b_2\|_{BMO} M_{t_2\delta}
 \big(T^1_{b_1-\lambda_1}(f_1,f_2)\big)(x)\\
&\le {C} \|b_2\|_{BMO} M_{\varepsilon}
 \big(T^1_{b_1-\lambda_1}(f_1,f_2)\big)(x)\\
&\le {C} \|b_2\|_{BMO} M_{\varepsilon}
 \big(T^1_{b_1}(f_1,f_2)\big)(x).
\end{split}
\end{equation}

Similarly, we have
\begin{equation}     \label{equ.sharp4}
\begin{split}
I_3 \le {C} \|b_1\|_{BMO} M_{\varepsilon}
\big(T^2_{b_2}(f_1,f_2)\big)(x).
\end{split}
\end{equation}

Now, we are in the position to consider the last term $I_4$. For each
$j=1,2$, we decompose $f_j$ as $f_j=f_j^0 + f_j^{\infty}$, where
$f_j^0=f_j\chi_{Q^*}$ and $f_j^{\infty} =f_j-f_j^0$. Let
$c=\sum_{j=1}^3 c_j$, where
\begin{align*}
&c_1=T((b_1-\lambda_1)f_1^0, (b_2-\lambda_2)f_2^{\infty})(x), \\
&c_2=T((b_1-\lambda_1)f_1^{\infty}, (b_2-\lambda_2)f_2^0)(x),\\
&c_3=T((b_1-\lambda_1)f_1^{\infty}, (b_2-\lambda_2)f_2^{\infty})(x).
\end{align*}

Then,
\begin{align*}
I_4 &\le {C}\bigg(\frac1{|Q|} \int_Q |T((b_1-\lambda_1)f_1^0,
 (b_2-\lambda_2)f_2^0)(z)|^{\delta}dz\bigg)^{1/{\delta}}\\
&\quad +{C}\bigg(\frac1{|Q|} \int_Q |T((b_1-\lambda_1)f_1^0,
 (b_2-\lambda_2)f_2^{\infty})(z)-c_1|^{\delta}dz\bigg)^{1/{\delta}}\\
&\quad +{C}\bigg(\frac1{|Q|} \int_Q |T((b_1-\lambda_1)f_1^{\infty},
 (b_2-\lambda_2)f_2^0)(z)-c_2|^{\delta}dz\bigg)^{1/{\delta}}\\
&\quad +{C}\bigg(\frac1{|Q|} \int_Q |T((b_1-\lambda_1)f_1^{\infty},
 (b_2-\lambda_2)f_2^{\infty})(z)-c_3|^{\delta}dz\bigg)^{1/{\delta}}\\
&:=I_{4,1} +I_{4,2} +I_{4,3} +I_{4,4}.
\end{align*}

We first estimate $I_{4,1}$. Applying Kolmgorov's inequality (Lemma
\ref{lem.kolm}) with $p=\delta<1/2$ and $q=1/2$, Theorem \ref{thm.A}
and (\ref{equ.bmo1}), we have
\begin{align*}
I_{4,1}&= C|Q|^{-1/{\delta}}\|T((b_1-\lambda_1)f_1^0,
 (b_2-\lambda_2)f_2^0)\|_{L^{\delta}(Q)}\\
&\le {C}|Q|^{-2}\|T((b_1-\lambda_1)f_1^0,
 (b_2-\lambda_2)f_2^0)\|_{L^{1/2,\infty}(Q)}\\
&\le {C}|Q|^{-2}\|T((b_1-\lambda_1)f_1^0,
 (b_2-\lambda_2)f_2^0)\|_{L^{1/2,\infty}(\rn)}\\
&\le {C}|Q|^{-2}\|(b_1-\lambda_1)f_1^0\|_{L^1(\rn)}
 \|(b_2-\lambda_2)f_2^0)\|_{L^1(\rn)}\\
&\le {C}\frac1{|Q|}\int_{Q^*}|b_1(z)-(b_1)_{Q^*}||f_1^0(z)|dz
 \frac1{|Q|}\int_{Q^*}|b_2(z)-(b_2)_{Q^*}||f_2^0(z)|dz\\
& \le{C}\|b_1\|_{BMO}\|b_2\|_{BMO} \|f_1\|_{L(\log{L}),Q^*}
  \|f_2\|_{L(\log{L}),Q^*}\\
&\le {C}\|b_1\|_{BMO}\|b_2\|_{BMO}
  {\mathcal{M}}_{L(\log{L})}(f_1,f_2)(x).
\end{align*}

Next, we consider the term $I_{4,2}$. For any $z\in{Q}$, we have
\begin{align*}
&|T((b_1-\lambda_1)f_1^0, (b_2-\lambda_2)f_2^{\infty})(z)-c_1|\\
&=|T((b_1-\lambda_1)f_1^0, (b_2-\lambda_2)f_2^{\infty})(z)
  -T((b_1-\lambda_1)f_1^0, (b_2-\lambda_2)f_2^{\infty})(x)|\\
&\le \int_{(\rn)^2} |K(z,y_1,y_2)-K(x,y_1,y_2)|
 |(b_1(y_1)-\lambda_1)f_1^0(y_1)|
   |(b_2(y_2)-\lambda_2)f_2^{\infty}(y_2)| dy_1dy_2\\
&\le \int_{Q^*}|(b_1(y_1)-\lambda_1)f_1^0(y_1)|
  \bigg(\int_{\rn\setminus{Q^*}} |K(z,y_1,y_2)-K(x,y_1,y_2)|
   |(b_2(y_2)-\lambda_2)f_2^{\infty}(y_2)|dy_2\bigg)dy_1
\end{align*}

Note the following fact that, for any $z\in{Q}$, $y_1\in {Q^*}$ and
$y_2\in{\mathcal{Q}}_k:={2^{k+3}\sqrt{n}Q}\setminus{2^{k+2}\sqrt{n}Q}$,
\begin{equation}      \label{equ.sharp5}
\begin{split}
|K(z,y_1,y_2)-K(x,y_1,y_2)| &\le \frac{C}{(|x-y_1|+|x-y_2|)^{2n}}
  \omega\bigg(\frac{|z-x|}{|x-y_1|+|x-y_2|}\bigg)\\
&\le {C}\frac{\omega(2^{-k})}{|2^{k+3}\sqrt{n}Q|^2},
\end{split}
\end{equation}
we get
\begin{align*}
&|T((b_1-\lambda_1)f_1^0, (b_2-\lambda_2)f_2^{\infty})(z)-c_1|\\
&\le {C}\int_{Q^*}|(b_1(y_1)-\lambda_1)f_1^0(y_1)|
  \bigg(\sum_{k=1}^{\infty} \int_{{\mathcal{Q}}_k}
  \frac{\omega(2^{-k})}{|2^{k+3}\sqrt{n}Q|^2}
   |(b_2(y_2)-\lambda_2)f_2^{\infty}(y_2)|dy_2\bigg)dy_1\\
&\le {C}\int_{Q^*}|(b_1(y_1)-\lambda_1)f_1^0(y_1)|
  \bigg(\sum_{k=1}^{\infty} \frac{\omega(2^{-k})}{|2^{k}Q^*|^2}
  \int_{2^kQ^*}|(b_2(y_2)-\lambda_2)f_2^{\infty}(y_2)|dy_2\bigg)dy_1\\
&\le {C}\sum_{k=1}^{\infty}\omega(2^{-k})
 \frac1{|2^{k}Q^*|}\int_{2^kQ^*}|(b_1(y_1)-\lambda_1)f_1(y_1)|dy_1
  \frac1{|2^{k}Q^*|}\int_{2^kQ^*}|(b_2(y_2)-\lambda_2)f_2(y_2)|dy_2.
\end{align*}

By (\ref{equ.bmo1}) and (\ref{equ.bmo2}), we have
\begin{equation}     \label{equ.sharp6}
\begin{split}
&\frac1{|2^{k}Q^*|} \int_{2^kQ^*}|(b_j(y_j)-\lambda_j)f_j(y_j)|dy_j\\
&=\frac1{|2^{k}Q^*|}\int_{2^kQ^*}|(b_j(y_j)-(b_j)_{Q^*})f_j(y_j)|dy_j\\
&\le \frac1{|2^{k}Q^*|}\int_{2^kQ^*}
  |(b_j(y_j)-(b_j)_{2^kQ^*})f_j(y_j)|dy_j\\
&\quad  +\frac{|(b_j)_{2^kQ^*}-(b_j)_{Q^*}|}{|2^{k}Q^*|}
  \int_{2^kQ^*}|f_j(y_j)|dy_j\\
& \le {C}\|b_j\|_{BMO}\|f_j\|_{L(\log{L}),2^kQ^*}
  +{C}k\|b_j\|_{BMO} \|f_j\|_{L(\log{L}),2^kQ^*}\\
& \le {C}k\|b_j\|_{BMO} \|f_j\|_{L(\log{L}),2^kQ^*}.
\end{split}
\end{equation}
Then by (\ref{equ.ldini-e}),
\begin{align*}
I_{4,2} &\le \frac{C}{|Q|} \int_Q |T((b_1-\lambda_1)f_1^0,
 (b_2-\lambda_2)f_2^{\infty})(z)-c_1| dz\\
&\le {C}\|b_1\|_{BMO}\|b_2\|_{BMO}
 \sum_{k=1}^{\infty}k^2\omega(2^{-k})
  \|f_1\|_{L(\log{L}),2^kQ^*}\|f_2\|_{L(\log{L}),2^kQ^*}\\
&\le {C}\|b_1\|_{BMO}\|b_2\|_{BMO}
 {\mathcal{M}}_{L(\log{L})}(f_1,f_2)(x).
\end{align*}

Similarly to $I_{4,2}$, we can estimate
$$I_{4,3} \le {C}\|b_1\|_{BMO}\|b_2\|_{BMO}
{\mathcal{M}}_{L(\log{L})}(f_1,f_2)(x).
$$

Finally, we consider the term $I_{4,4}$.  For any $z\in{Q}$ and
$(y_1,y_2)\in (2^{k+3}\sqrt{n} Q)^2 \setminus (2^{k+2}\sqrt{n} Q)^2$,
similar to (\ref{equ.sharp5}) we have
$$|K(z,y_1,y_2)-K(x,y_1,y_2)| \le {C}\frac{\omega(2^{-k})}{|2^{k+3}\sqrt{n}Q|^2}.
$$
This together with (\ref{equ.sharp6}) gives \begin{align*}
&|T((b_1-\lambda_1)f_1^{\infty},(b_2-\lambda_2)f_2^{\infty})(z)-c_3|\\
&\le \int_{(\rn\setminus{Q^*})^2}|K(z,y_1,y_2)-K(x,y_1,y_2)|
 \bigg(\prod_{j=1}^2|(b_j(y_j)-\lambda_j)f_j^{\infty}(y_j)|\bigg)dy_1dy_2 \\
&\le \sum_{k=1}^{\infty}\int_{(2^{k+3}\sqrt{n} Q)^2 \setminus (2^{k+2}\sqrt{n} Q)^2}
  |K(z,y_1,y_2)-K(x,y_1,y_2)|
   \bigg(\prod_{j=1}^2|(b_j(y_j)-\lambda_j)f_j^{\infty}(y_j)|\bigg)dy_1dy_2\\
&\le {C}\sum_{k=1}^{\infty}
 \int_{(2^{k+3}\sqrt{n} Q)^2} \frac{\omega(2^{-k})}{|2^{k+3}\sqrt{n}Q|^2}
   \bigg(\prod_{j=1}^2|(b_j(y_j)-\lambda_j)f_j(y_j)|\bigg)dy_1dy_2\\
&\le {C}\|b_1\|_{BMO}\|b_2\|_{BMO}
 \sum_{k=1}^{\infty}k^2\omega(2^{-k})
  \|f_1\|_{L(\log{L}),2^kQ^*}\|f_2\|_{L(\log{L}),2^kQ^*}\\
&\le {C}\|b_1\|_{BMO}\|b_2\|_{BMO}
   {\mathcal{M}}_{L(\log{L})}(f_1,f_2)(x),
\end{align*}
which concludes that
\begin{align*}
I_{4,4}&=C\bigg(\frac1{|Q|} \int_Q |T((b_1-\lambda_1)f_1^{\infty},
 (b_2-\lambda_2)f_2^{\infty})(z)-c_3|^{\delta}dz\bigg)^{1/{\delta}}\\
&\le{C}\frac1{|Q|} \int_Q |T((b_1-\lambda_1)f_1^{\infty},
 (b_2-\lambda_2)f_2^{\infty})(z)-c_3|dz\\
&\le {C}\|b_1\|_{BMO}\|b_2\|_{BMO}
   {\mathcal{M}}_{L(\log{L})}(f_1,f_2)(x).
\end{align*}
This, together with (\ref{equ.sharp1}), (\ref{equ.sharp2}),
(\ref{equ.sharp3}) and  (\ref{equ.sharp4}), gives
\begin{align*}
M^{\sharp}_{\delta}\big(T_{\Pi\vec{b}}(\vec{f})\big)(x)
 &\le{C}\|b_1\|_{BMO}\|b_2\|_{BMO}
  \Big\{{\mathcal{M}}_{L(\log{L})}(f_1,f_2)(x)
   +M_{\varepsilon}\big(T(f_1,f_2)\big)(x)\Big\}\\
&\qquad + C\|b_2\|_{BMO}
M_{\varepsilon}\big(T_{b_1}^1(f_1,f_2)\big)(x)
  + C\|b_1\|_{BMO} M_{\varepsilon}\big(T_{b_2}^2(f_1,f_2)\big)(x).
\end{align*}

The proof of Proposition \ref{equ.sharp0} is now complete.
\end{proof}

\begin{remark}   \label{remark.sharp}
We can also obtain analogous estimates to (\ref{equ.sharp0})
for iterated commutators involving $j<m$ functions in $BMO$.
More precisely, for $\sigma=\{\sigma(1), \cdots, \sigma(j)\}$,
we have
\begin{equation}   \label{equ.sharp0.1}
\begin{split}
M_{\delta}^{\sharp}\big(T_{\Pi{\vec{b}}_{\sigma}}(\vec{f})\big)(x)
&\le{C} \bigg(\prod_{k=1}^{j} \|b_k\|_{BMO}\bigg)
\Big\{{\mathcal{M}_{L(\log{L})}}(\vec{f})(x)
+M_{\varepsilon}\big(T(\vec{f})\big)(x) \Big\}\\
&\quad +C\sum_{k=1}^{j-1} \sum_{\eta \in {C_k^m}}
\bigg(\prod_{i=1}^k \|b_{\eta(i)}\|_{BMO}\bigg)
M_{\varepsilon}\big(T_{\Pi\vec{b}_{\eta'}}(\vec{f})\big)(x).
\end{split}
\end{equation}
\end{remark}

From the pointwise estimates obtained above, we can get the
following strong and weak type estimates for the iterated commutator
$T_{\Pi\vec{b}}$.

\begin{proposition}  \label{prop.2}  
Let $T$ be an $m$-linear $\omega$-CZO with $\omega$ satisfying
(\ref{equ.ldini}) and $\vec{b}\in {BMO^m}$. Suppose that
$0<p<\infty$ and $w\in {A_{\infty}}$. Then, there exist positive
constant $C_w$ (depending on the $A_{\infty}$ constant of $w$, but
independent of $\vec{b}$) and $C_{w,\vec{b}}$ (depending on $w$ and
$\vec{b}$) such that
\begin{equation}  \label{equ.prop-1}
\int_{\rn}\big|T_{\Pi\vec{b}}(\vec{f})(x)\big|^pw(x)dx \le {C_w}
\bigg(\prod_{j=1}^m\|b_j\|_{BMO}\bigg)^p
 \int_{\rn} \big[\mathcal{M}_{L(\log{L})}(\vec{f})(x)\big]^p w(x)dx
\end{equation}
and
\begin{equation}  \label{equ.prop-2}
\begin{split}
\sup_{t>0} & \frac1{\Phi^{(m)}(1/t)} w\big(\big\{y\in \rn:
\big|T_{\Pi\vec{b}}(\vec{f})(y)\big|>t^m\big\}\big) \\
 & \le {C_{w,\vec{b}}}  \sup_{t>0}\frac1{\Phi^{(m)}(1/t)} w\big(\big\{y\in \rn:
  \mathcal{M}_{L(\log{L})}(\vec{f})(y)>t^m\big\}\big),
 \end{split}
\end{equation}
for all bounded measurable functions $f_1, \cdots, f_m$ with
compact support.
\end{proposition}

To prove Proposition \ref{prop.2}, we need the following
results obtained in \cite{lz}.

\begin{lemma}[\cite{lz} Theorem 6.2]   \label{lem.lz.th6.2}
Let $T$ be an $m$-linear $\omega$-CZO with $\omega \in{Dini(1)}$,
$0<p<\infty$ and $w\in {A_{\infty}}$. Then there exists a constant
$C>0$ such that
$$\|T(\vec{f})\|_{L^p(w)}\le {C}\|{\cal{M}}(\vec{f})\|_{L^p(w)}
$$
for all bounded measurable functions $f_1, \cdots, f_m$
with compact support.
\end{lemma}

We remark that the authors of \cite{lz} proved Lemma \ref{lem.lz.th6.2}
for $1/m\le {p}<\infty$. Indeed, we can extend the range of $p$ from
$1/m\le {p}<\infty$ to $0<p<\infty$ by using the same argument of the
proof of Corollary 3.8 in \cite{loptt}.

\begin{lemma}[\cite{lz} Theorem 7.2]   \label{lem.lz.th7.2}
Let $T$ be an $m$-linear $\omega$-CZO and $\vec{b}\in {BMO^m}$.
Suppose that $1\le {j}\le {m}$ is an integer and
$T_{b_j}^{j}$ is the $j$-th entry commutator of $T$ with
$\vec{b}$ defined by (\ref{equ.j-th}). If $0<p<\infty$,
$w\in {A_{\infty}}$ and $\omega$ satisfies
$$\int_0^1\frac{\omega(t)}{t}\bigg(1+\log\frac1{t}\bigg)dt<\infty,
$$
then, there exists a constant $C>0$, depending on the $A_{\infty}$
constant of $w$, such that
$$\int_{\rn}\big|T_{b_j}^{j}(\vec{f})(x)\big|^pw(x)dx \le {C}
\|b_j\|_{BMO}^p
\int_{\rn} \big[\mathcal{M}^j_{L(\log{L})}(\vec{f})(x)\big]^p w(x)dx
$$
for all bounded measurable functions $f_1, \cdots, f_m$
with compact support.
\end{lemma}

\begin{proof}[Proof of Proposition \ref{prop.2}]
Similar to the proof of Theorem 3.19 in \cite{loptt}, see also the proof of
Theorem 3.2 in \cite{pptt}, we can get (\ref{equ.prop-1}) and (\ref{equ.prop-2}).
For the sake of completeness, we give the proof of (\ref{equ.prop-1})
and indicate the argument in the case $m=2$.
An iterative procedure using (\ref{equ.sharp0}) and (\ref{equ.sharp0.1})
can be followed to obtain the general case $m>2$.

To prove (\ref{equ.prop-1}), we may assume
$\|{\cal{M}}_{L(\log{L})}(\vec{f})\|_{L^p(w)}$
is finite since otherwise there is nothing to be proven.
If we can verify that, for $0<\delta<1/m$ small enough,
\begin{equation} \label{equ.claim}
\int_{\rn}\big[M_{\delta}(T_{\Pi\vec{b}}(\vec{f}))(x)\big]^pw(x)dx<\infty,
\end{equation}
then, by Lemma \ref{lem.fs} (1) and Proposition \ref{prop.sharp} with exponents
$0<\delta<\varepsilon<1/m$, we have
\begin{equation}  \label{equ.prop.2-1}
\begin{split}
\big\|T_{\Pi\vec{b}}(\vec{f})\big\|_{L^p(w)}
 & \le {C} \big\|M_{\delta}\big(T_{\Pi\vec{b}}(\vec{f})\big)\big\|_{L^p(w)}
 \le {C} \big\|M_{\delta}^{\sharp}\big(T_{\Pi\vec{b}}(\vec{f})\big)\big\|_{L^p(w)}\\
 &\le{C} \|b_1\|_{BMO}\|b_2\|_{BMO} \big\|{\mathcal{M}_{L(\log{L})}}(f_1,f_2)\big\|_{L^p(w)}\\
&\quad +{C} \|b_1\|_{BMO}\|b_2\|_{BMO}\big\|M_{\varepsilon}\big(T(f_1,f_2)\big)\big\|_{L^p(w)}\\
&\quad +C\|b_2\|_{BMO}\big\|M_{\varepsilon}\big(T^1_{b_1}(f_1,f_2)\big)\big\|_{L^p(w)}\\
  &\quad +C\|b_1\|_{BMO}\big\|M_{\varepsilon}\big(T^2_{b_2}(f_1,f_2)\big)\big\|_{L^p(w)}.
\end{split}
\end{equation}

Since $w\in {A_{\infty}}$ then $w\in {A_{p_0}}$ for some $p_0>1$.
Choose $\varepsilon$ small enough so that $p/\varepsilon>p_0$,
then $w\in {A_{p/\varepsilon}}$. Applying the $L^{p/\varepsilon}(w)$-boundedness of
$M$ and Lemma \ref{lem.lz.th6.2}, we get
\begin{equation*}
\begin{split}
\big\|M_{\varepsilon}\big(T(f_1,f_2)\big)\big\|_{L^p(w)}
&= \big\|M\big(|T(f_1,f_2)|^{\varepsilon}\big)\big\|^{1/{\varepsilon}}_{L^{p/{\varepsilon}}(w)}\\
&\le{C}\big\||T(f_1,f_2)|^{\varepsilon}\big\|^{1/{\varepsilon}}_{L^{p/{\varepsilon}}(w)}\\
&= {C}\big\|T(f_1,f_2)\|_{L^{p}(w)}\\
& \le {C} \big\|{\cal{M}}(f_1,f_2)\big\|_{L^p(w)}\\
& \le {C} \big\|{\cal{M}}_{L(\log{L})}(f_1,f_2)\big\|_{L^p(w)}.
\end{split}
\end{equation*}

Similarly, by the $L^{p/\varepsilon}(w)$-boundedness of
$M$ and Lemma \ref{lem.lz.th7.2}, we have
$$\big\|M_{\varepsilon}\big(T^j_{b_j}(f_1,f_2)\big)\big\|_{L^p(w)}
\le {C} \|b_j\|_{BMO}\big\|{\cal{M}}_{L(\log{L})}(f_1,f_2)\big\|_{L^p(w)},
\quad j=1,2.
$$
Then, (\ref{equ.prop-1}) follows from (\ref{equ.prop.2-1}) and
the above estimates.

Thus, in order to finish the proof of Proposition \ref{prop.2},
it remains to verify (\ref{equ.claim}).
Note that since $w$ in $A_{\infty}$, $w$ is also in $A_{p_0}$ for
some $p_0$ satisfying
$0<\max\{1,mp\}<p_0<\infty$. So, for any $0<\delta<p/p_0<1/m$, it follows
from the fact that $M$ is bounded from $L^{p_0}(w)$ to itself that
\begin{equation*}
\begin{split}
\big\|M_{\delta}\big(T_{\Pi\vec{b}}(\vec{f})\big)\big\|_{L^p(w)}
 &\le \big\|M_{p/p_0}\big(T_{\Pi\vec{b}}(\vec{f})\big)\big\|_{L^p(w)}
   =\big\|M\big(|T_{\Pi\vec{b}}(\vec{f})|^{p/p_0}\big)\big\|_{L^{p_0}(w)}^{p_0/p}\\
 &\le{C} \big\||T_{\Pi\vec{b}}(\vec{f})|^{p/p_0}\big\|_{L^{p_0}(w)}^{p_0/p}
  =C\big\|T_{\Pi\vec{b}}(\vec{f})\big\|_{L^{p}(w)}.
\end{split}
\end{equation*}

Then, it is enough to show that $\|T_{\Pi\vec{b}}(\vec{f})\|_{L^{p}(w)}$ 
is finite for each tuple $\vec{f}=(f_1,f_2)$ of bounded functions with
compact support for which $\|{\cal{M}}_{L\log{L}}(\vec{f})\|_{L^p(w)}$ 
is finite.

Firstly, we assume that $b_1, b_2 \in {L^{\infty}(\rn)}$.
Let $B(0,r)$ be a sufficiently large ball centered at the origin 
with radius $r$ such that $\supp{f_j}\subset {B(0,r/2)}$ for $j=1,2$. 
Then
$$\big\|T_{\Pi\vec{b}}(\vec{f})\big\|_{L^{p}(w)}^p
  =\bigg(\int_{B(0,r)}+\int_{\rn\setminus{B(0,r)}}\bigg)
   \big|T_{\Pi\vec{b}}(\vec{f})(x)\big|^pw(x)dx
     =: I_1 + I_2.
$$

By (\ref{equ.sharp-T}) and Lemma \ref{lem.lz.th6.2} we have
\begin{equation*}
\begin{split}
I_1& \le{C} \|b_1\|_{L^{\infty}} \|b_2\|_{L^{\infty}} \big\|T(f_1,f_2)\big\|_{L^p(w)}^p
   +C\|b_2\|_{L^{\infty}}\big\|T(b_1f_1,f_2)\big\|_{L^p(w)}^p \\
  & \quad +C\|b_1\|_{L^{\infty}}\big\|T(f_1,b_2f_2)\big\|_{L^p(w)}^p
    +C\big\|T(b_1f_1,b_2f_2)\big\|_{L^p(w)}^p\\
 & \le{C} \|b_1\|_{L^{\infty}} \|b_2\|_{L^{\infty}} \big\|{\cal{M}}(f_1,f_2)\big\|_{L^p(w)}^p
   +C\|b_2\|_{L^{\infty}}\big\|{\cal{M}}(b_1f_1,f_2)\big\|_{L^p(w)}^p \\
  & \quad +C\|b_1\|_{L^{\infty}} \big\|{\cal{M}}(f_1,b_2f_2)\big\|_{L^p(w)}^p
    +C\big\|{\cal{M}}(b_1f_1,b_2f_2)\big\|_{L^p(w)}^p\\
 & \le{C} \|b_1\|_{L^{\infty}} \|b_2\|_{L^{\infty}} \big\|{\cal{M}}(f_1,f_2)\big\|_{L^p(w)}^p\\
 & \le{C} \big\|{\cal{M}}_{L\log{L}}(f_1,f_2)\big\|_{L^p(w)}^p <\infty.
\end{split}
\end{equation*}

On the other hand, we have for $b_1, b_2 \in L^{\infty}(\rn)$
and $x\notin {B(0,r)}$,
\begin{equation*}
\begin{split}
\big|T_{\Pi\vec{b}}(\vec{f})(x)\big| & =\bigg| \int_{(\rn)^2}
  (b_1(x)-b_1(y_1))(b_2(x)-b_2(y_2))K(x,\vec{y})f_1(y_1)f_2(y_2)d{\vec{y}}\bigg| \\
 & \le {C} \|b_1\|_{L^{\infty}} \|b_2\|_{L^{\infty}}
   \int_{(\rn)^2} \frac{|f_1(y_1)||f_2(y_2)|}{|x-y_1|^n|x-y_2|^n} dy_1 dy_2\\
 & \le {C} \|b_1\|_{L^{\infty}} \|b_2\|_{L^{\infty}}
   \int_{\supp{f_1}} \frac{|f_1(y_1)|}{|x-y_1|^n} dy_1
    \int_{\supp{f_2}} \frac{|f_2(y_2)|}{|x-y_2|^n} dy_2\\
 & \le {C} \|b_1\|_{L^{\infty}} \|b_2\|_{L^{\infty}}
  \frac{1}{|x|^n} \int_{B(0,2|x|)}|f_1(y_1)| dy_1
    \frac{1}{|x|^n} \int_{B(0,2|x|)} |f_2(y_2)| dy_2\\
 & \le {C} \|b_1\|_{L^{\infty}} \|b_2\|_{L^{\infty}} {\cal{M}}(f_1,f_2)(x)\\
 & \le {C} \|b_1\|_{L^{\infty}} \|b_2\|_{L^{\infty}} {\cal{M}}_{L\log{L}}(\vec{f})(x).
\end{split}
\end{equation*}
Thus
$$ I_2 \le {C} \|b_1\|_{L^{\infty}} \|b_2\|_{L^{\infty}}
 \int_{\rn\setminus{B(0,r)}} \big[{\cal{M}}_{L\log{L}}(\vec{f})(x)\big]^pw(x)dx,
$$
which is finite since we are assuming that
$\|{\cal{M}}_{L\log{L}}(\vec{f})\|_{L^p(w)}$ is finite.
This proves (\ref{equ.prop-1}) provided $b_1$ and $b_2$ are bounded.

To obtain the result for  the case $\vec{b}=(b_1,b_2) \in {BMO^2}$,
we will use a limiting argument as in \cite[page 1254]{loptt}.
Let $k$ be any positive integer and $\vec{b}^k=(b_1^k,b_2^k)$,
where $b_j^k$ is the following truncation of $b_j$ by $k$, which is
given by
$$b_j^k(x)=\begin{cases}
k, \quad & b_j(x)>k,\\
b_j(x), & |b_j(x)|\le {k},\\
-k, & b_j(x)<-k,
\end{cases}     \quad j=1,2.
$$

Clearly, $\{b_j^k\}_{k=1}^{\infty}$ converges pointwise to
$b_j$ as $k\to \infty$ and $\|b_j^k\|_{BMO}\le {C}\|b_j\|_{BMO}$
with $C$ being a positive constant independent of $k$.

Now, taking into account the fact that $f_1$ and $f_2$ are
bounded functions with compact support, we deduce that
$\big\{T(b_1^kf_1,b_2^kf_2)\big\}$ converges to $T(b_1f_1,b_2f_2)$
in any $L^q(\rn)$ for $q>1$ as $N\to \infty$ since $T$ is
bounded on $L^q(\rn)$. Then there is a subsequence of
$\{\vec{b}^{k}\}$, $\{\vec{b}^{k'}\}=\{(b_1^{k'},b_2^{k'})\}$,
so that $\big\{T(b_1^{k'}f_1,b_2^{k'}f_2)\big\}$ converges to
$T(b_1f_1,b_2f_2)$ almost everywhere as $k'\to \infty$.

Similarly, we can choose a subsequence of $\{\vec{b}^{k'}\}$,
$\{\vec{b}^{k''}\}=\{(b_1^{k''},b_2^{k''})\}$, so that
$\big\{T(b_1^{k''}f_1,f_2)\big\}$ and $\big\{T(f_1,b_2^{k''}f_2)\big\}$
converge to $T(b_1f_1,f_2)$ and $T(f_1,b_2f_2)$, respectively,
almost everywhere as $k''\to \infty$.
Thus, it follows from (\ref{equ.sharp-T}) that $T_{\Pi{\vec{b}}^{k''}}(f_1,f_2)$
converges to $T_{\Pi\vec{b}}(f_1,f_2)$ almost everywhere as $k''\to \infty$.

Observe that (\ref{equ.prop-1}) holds for $T_{\Pi{\vec{b}}^{k''}}(f_1,f_2)$
since $b_1^{k''}$ and $b_2^{k''}$ are bounded, then
\begin{equation*}
\begin{split}
\int_{\rn}\big|T_{\Pi\vec{b}^{k''}}(\vec{f})(x)\big|^pw(x)dx
 &\le {C_w}\|b_1^{k''}\|_{BMO}^p \|b_2^{k''}\|_{BMO}^p
   \int_{\rn} \big[\mathcal{M}_{L(\log{L})}(\vec{f})(x)\big]^p w(x)dx\\
 &\le {C_w}\|b_1\|_{BMO}^p \|b_2\|_{BMO}^p
   \int_{\rn} \big[\mathcal{M}_{L(\log{L})}(\vec{f})(x)\big]^p w(x)dx.
\end{split}
\end{equation*}
This together with Fatou's lemma gives (\ref{equ.prop-1}) for
the case $\vec{b}\in {BMO^2}$.

For the general case $m>2$, an iterative using (\ref{equ.sharp0}) and
(\ref{equ.sharp0.1}) leads us to estimate
$$\sum_{j=1}^{m-1} \sum_{\sigma \in {C_j^m}}
\bigg(\prod_{i=1}^j \|b_{\sigma(i)}\|_{BMO}\bigg)
M_{\varepsilon}\big(T_{\Pi\vec{b}_{\sigma'}}(\vec{f})\big)(x)$$
by a linear combination of the terms like $M_{\varepsilon}(T^j_{b_j}(\vec{f}))(x)$,
${\mathcal{M}_{L(\log{L})}}(\vec{f})(x)$ and $M_{\varepsilon}\big(T(\vec{f})\big)(x)$.
So, the procedure used above can be followed to get (\ref{equ.prop-1}) for
the general case $m>2$.

We note that one can prove (\ref{equ.prop-2}) by following the second part
of the proof of Theorem 3.19 in \cite{loptt} step by step, see also the
proof of Theorem 3.2 in \cite{pptt}. We omit the details.

The proof of Proposition \ref{prop.2} is complete.
\end{proof}

\begin{proof}{\bf {of Theorem \ref{thm.1.1}}}\ \  It suffices to prove
Theorem \ref{thm.1.1} for $f_1, \cdots, f_m$ being bounded with compact
support. For $\vec{w} \in {A_{\vec{P}}}$, Lemma \ref{lem.ww}
implies $\nu_{\vec{w}} \in {A_{\infty}}$. It follows from (\ref{equ.prop-1}) that
$$\big\|T_{\Pi\vec{b}}(\vec{f})\big\|_{L^p(\nu_{\vec{w}})} \le {C}
 \bigg(\prod_{j=1}^m\|b_j\|_{BMO}\bigg)
  \big\|{\cal{M}}_{L(\log{L})}(\vec{f})\big\|_{L^p(\nu_{\vec{w}})}.
$$

By (\ref{equ.mmm}) and Lemma \ref{lem.mrw}, for some $r>1$,
\begin{equation*}
\begin{split}
\big\|T_{\Pi\vec{b}}(\vec{f})\big\|_{L^p(\nu_{\vec{w}})} &\le {C}
 \bigg(\prod_{j=1}^m\|b_j\|_{BMO}\bigg)
  \big\|{\cal{M}}_{r}(\vec{f})\big\|_{L^p(\nu_{\vec{w}})}\\
&\le {C}  \bigg(\prod_{j=1}^m\|b_j\|_{BMO}\bigg)
  \prod_{j=1}^m \|f_j\|_{L^{p_j}(w_j)}.
\end{split}
\end{equation*}
This concludes the proof of Theorem \ref{thm.1.1}.
\end{proof}

\begin{proof}{\bf {of Theorem \ref{thm.1.2}}}\ \ By homogeneity,
we can assume $\lambda=1$ and $\|b_j\|_{BMO}=1$ for $j=1,\cdots,m$,
and hence it is enough to prove
$$\nu_{\vec{w}}\big(\big\{ x\in \rn:
 \big|T_{\Pi\vec{b}}(\vec{f})(x)\big|>1 \big\}\big) \le {C}
  \prod_{j=1}^m \left(\int_{\rn}\Phi^{(m)} (|f_j(x)|)w_j(x)dx\right)^{1/m}.
$$

Note that $\nu_{\vec{w}}\in {A_{\infty}}$ when $\vec{w} \in {A_{(1,\cdots,1)}}$,
it follows from (\ref{equ.prop-2}) and Lemma \ref{lem.m-llogl} that
\begin{equation*}
\begin{split}
\nu_{\vec{w}}&\big(\big\{ x\in \rn:
 \big|T_{\Pi\vec{b}}(\vec{f})(x)\big|>1 \big\}\big)\\
 & \le \sup_{t>0}\frac1{\Phi^{(m)}(1/t)} \nu_{\vec{w}}
  \big(\big\{y\in \rn: \big|T_{\Pi\vec{b}}(\vec{f})(y)\big|>t^m\big\}\big) \\
 & \le {C} \sup_{t>0}\frac1{\Phi^{(m)}(1/t)} \nu_{\vec{w}}
  \big(\big\{y\in \rn:  \mathcal{M}_{L(\log{L})}(\vec{f})(y)>t^m\big\}\big)\\
 &\le {C} \sup_{t>0}\frac1{\Phi^{(m)}(1/t)}
  \prod_{j=1}^m \left(\int_{\rn}\Phi^{(m)}
   \bigg(\frac{|f_j(x)|}{t} \bigg)w_j(x)dx\right)^{1/m}\\
 &\le {C} \sup_{t>0}\frac1{\Phi^{(m)}(1/t)}
  \prod_{j=1}^m \left(\int_{\rn}\Phi^{(m)}
   (|f_j(x)|)w_j(x) \Phi^{(m)}(1/t) dx\right)^{1/m}\\
 &\le {C} \prod_{j=1}^m \left(\int_{\rn}\Phi^{(m)}
   (|f_j(x)|)w_j(x) dx\right)^{1/m}.
\end{split}
\end{equation*}

So complete the proof of Theorem \ref{thm.1.2}.
\end{proof}


\section{Proof of Theorem \ref{thm.1.3}}   \label{variable}

We first recall some facts on variable exponent Lebesgue spaces.
As the classical Lebesgue norm, the (quasi-)norm of variable
exponent Lebesgue space is also homogeneous in the exponent.
Precisely, we have the following result, see \cite[Lemma 2.3]{cu-w2014}.

\begin{lemma}  \label{lem.s-norm}
Given $p(\cdot) \in \mathscr{P}_0$, then for all $s>0$, we have
$$\big\||f|^s\big\|_{L^{p(\cdot)}(\rn)}=\|f\|^s_{L^{sp(\cdot)}(\rn)}.
$$
\end{lemma}

The following generalized H\"older's inequality was proved in
\cite[Theorem 2.3]{h-x}. 

\begin{lemma}  \label{lem.gholder}
Let $r(\cdot),~ p(\cdot),~ q(\cdot)\in \mathscr{P}_0$ so that
$1/r(\cdot)=1/p(\cdot) + 1/q(\cdot)$. Then, for any $f\in {L^{p(\cdot)}}(\rn)$
and $g\in {L^{q(\cdot)}}(\rn)$,
$$ \|fg\|_{L^{r(\cdot)}(\rn)} \le {C}
 \|f\|_{L^{p(\cdot)}(\rn)} \|g\|_{L^{q(\cdot)}(\rn)},
$$
where the constant $C$ depends only on $p(\cdot)$ and $q(\cdot)$.
\end{lemma}

By induction argument, we can generalize Lemma \ref{lem.gholder}
to more exponents.

\begin{lemma} \label{lem.gholder-m}
Let $q(\cdot),~ q_{1}(\cdot),\cdots, q_{m}(\cdot)\in
\mathscr{P}_0$ so that $1/{q(\cdot)}=1/{q_{1}(\cdot)} +\cdots +1/{q_{m}(\cdot)}$.
Then for any $f_j\in {L^{q_{j}(\cdot)}}(\rn)$, $j=1, \cdots,m$,
$$ \|f_1\cdots {f_m}\|_{L^{q(\cdot)}(\rn)} \le {C}
\|f_1\|_{L^{q_{1}(\cdot)}(\rn)}\cdots \|f_m\|_{L^{q_{m}(\cdot)}(\rn)},
$$
where the constant $C$ depends only on $q_1(\cdot), \cdots, q_m(\cdot)$.
\end{lemma}

In 2012, Cruz-Uribe, Fiorenza and Neugebauer \cite{cfn} proved
the following weighted norm inequality for Hardy-Littlewood
maximal operator on weighted variable exponent Lebesgue
space, see \cite[Theorem 1.5]{cfn}.

\begin{lemma}[\cite{cfn}] \label{lem.c-f-n}
Let $p(\cdot) \in \mathscr{P}$ and satisfy (\ref{equ.log1})
and (\ref{equ.log2}), then for any $v\in {A_{p(\cdot)}}$,
$$\|Mf\|_{L_v^{p(\cdot)}(\rn)} \le {C} \|f\|_{L_v^{p(\cdot)}(\rn)}.
$$
\end{lemma}

The following extrapolation theorem is due to Cruz-Uribe
and Wang \cite[Theorem 2.24]{cu-w}.

\begin{lemma}[\cite{cu-w}] \label{lem.extra}
Given a family $\mathcal{F}$ of ordered pairs of measurable
functions. Suppose that for some $0<p_0<\infty$ and
every $w\in {A_{\infty}}$, the following inequality holds for
all $(f,g)\in \mathcal{F}$,
$$\int_{\rn} |f(x)|^{p_0}\omega(x)dx \le {C_0}
\int_{\rn} |g(x)|^{p_0}\omega(x) dx.
$$
Let $p(\cdot)\in \mathscr{P}_0$, if there exists $s\le {p^{-}}$
such that $v^s \in {A_{p(\cdot)/s}}$ and $M$ is bounded on
$L_{v^{-s}}^{(p(\cdot)/s)'}(\rn)$, then there is a positive constant
$C$ such that
$$\|f\|_{L_v^{p(\cdot)}(\rn)}\le {C}\|g\|_{L_v^{p(\cdot)}(\rn)}
$$
for every pair $(f,g)\in \mathcal{F}$ such that the left-hand
side is finite.
\end{lemma}

To prove our result, we will need the following density property,
see \cite[Lemma 3.1]{cu-w}.

\begin{lemma}[\cite{cu-w}]    \label{lem.density}
Given $p(\cdot)\in \mathscr{P}$ and a weight
$v \in L_{\mathrm{loc}}^{p(\cdot)}(\rn)$, then $L_c^{\infty}(\rn)$,
the set of all bounded functions with compact support, is
dense in $L_v^{p(\cdot)}(\rn)$.
\end{lemma}

The following monotone convergence theorem in variable
exponent Lebesgue spaces is due to \cite[Lemma 2.5]{cu-w2014}.

\begin{lemma}[\cite{cu-w2014}] \label{lem.monotone}
Suppose $p(\cdot) \in  \mathscr{P}_0$. Given a sequence $\{f_k\}$
of $L^{p(\cdot)}(\rn)$ functions that increases pointwise almost
everywhere to a function $f$, we have
$$\lim_{k\to\infty}\|f_k\|_{L^{p(\cdot)}(\rn)}
   =\|f\|_{L^{p(\cdot)}(\rn)}.
$$
\end{lemma}

We also need the following result for weights with variable exponents.

\begin{lemma} \label{lem.weight}
Let $p(\cdot) \in \mathscr{P}_0$ and $p_1(\cdot), \cdots,
p_m(\cdot) \in \mathscr{P}$ with ${1}/{p(\cdot)}
={1}/{p_1(\cdot)} +\cdots +{1}/{p_m(\cdot)}$.
For $v_j \in {A_{p_j(\cdot)}}$, $j=1,\cdots,m$, let $v=\prod_{j=1}^m v_j$.
Then $v^{1/m} \in {A_{mp(\cdot)}}$.
\end{lemma}

\begin{proof}
Since $p(\cdot) \in \mathscr{P}_0$, $p_1(\cdot), \cdots, p_m(\cdot)
\in \mathscr{P}$ and
${1}/{p(\cdot)} ={1}/{p_1(\cdot)} +\cdots +{1}/{p_m(\cdot)}$
then $mp(\cdot) \in \mathscr{P}$ and
$$\frac1{(mp(\cdot))'}=\sum_{j=1}^m\frac1{mp'_j(\cdot)}.
$$

By the generalized H\"older's inequality (Lemma \ref{lem.gholder-m}) and
Lemma \ref{lem.s-norm}, we have
\begin{equation*}
\begin{split}
&|B|^{-1} \big\|v^{1/m}\chi_B \big\|_{L^{mp(\cdot)}(\rn)}
    \big\|v^{-1/m}\chi_B\big\|_{L^{(mp(\cdot))'}(\rn)}\\
&\le {C}|B|^{-1} \prod_{j=1}^m\big\| v_j^{1/m}\chi_B\big\|_{L^{mp_j(\cdot)}(\rn)}
   \prod_{j=1}^m \big\|v_j^{-1/m}\chi_B\big\|_{L^{mp'_j(\cdot)}(\rn)}\\
&\le {C}|B|^{-1} \prod_{j=1}^m\big\| v_j\chi_B\big\|^{1/m}_{L^{p_j(\cdot)}(\rn)}
   \prod_{j=1}^m \big\|v_j^{-1}\chi_B\big\|^{1/m}_{L^{p'_j(\cdot)}(\rn)}\\
&={C}\prod_{j=1}^m\bigg(|B|^{-1} \big\| v_j\chi_B\big\|_{L^{p_j(\cdot)}(\rn)}
   \big\|v_j^{-1}\chi_B\big\|_{L^{p'_j(\cdot)}(\rn)}\bigg)^{1/m}\\
&<\infty,
\end{split}
\end{equation*}
where the last step follows from $v_j \in {A_{p_j(\cdot)}}$, $j=1,\cdots,m$.
This concludes the proof.
\end{proof}

Now, we have all the ingredients to prove Theorem \ref{thm.1.3}.

\begin{proof}{\bf {of Theorem \ref{thm.1.3}}}\ \
By Lemma \ref{lem.density} it suffices to prove Theorem
\ref{thm.1.3} for all bounded functions $f_1,\cdots,f_m$
with compact support.

We define a sequence of operators $\{T_k\}_{k \in \mathbb{N}}$,
where $T_k(\vec{f})=\min\{|T_{\Pi\vec{b}}(\vec{f})|,k\}\chi_{B(0,k)}$
and $T_k(\vec{f})=0$ when $x\notin {B(0,k)}$,
and consider the following family
$$\mathcal{F}=\big\{\big(T_k(\vec{f}),
{\cal{M}}_{L(\log{L})}(\vec{f})\big): \vec{f}=(f_1, \cdots, f_m)
\in \big(L_c^{\infty}(\rn)\big)^m, k\in\mathbb{N}\big\}.
$$

It follows from Proposition \ref{prop.2} that, for any
$0<p_0<\infty$ and every $w\in {A_{\infty}}$,
$$\int_{\rn}\big|T_{\Pi\vec{b}}(\vec{f})(x)\big|^{p_0}w(x)dx
\le {C}\int_{\rn} \big[{\cal{M}}_{L(\log{L})}(\vec{f})(x)
\big]^{p_0}w(x)dx
$$
holds for all bounded functions $f_1,\cdots,f_m$ with compact
support.

Since $0\le T_k(\vec{f})(x) \le |T_{\Pi\vec{b}}(\vec{f})(x)|$,
then for any $0<p_0<\infty$ and every $w\in {A_{\infty}}$,
$$\int_{\rn}\big[T_k(\vec{f})(x)\big]^{p_0}w(x)dx
\le {C}\int_{\rn} \big[{\cal{M}}_{L(\log{L})}(\vec{f})(x)
\big]^{p_0}w(x)dx
$$
holds for every ordered pair
$\big(T_k(\vec{f}), {\cal{M}}_{L(\log{L})}(\vec{f})\big)$
in $\mathcal{F}$.

By Lemma \ref{lem.weight} we have $v^{1/m} \in {A_{mp(\cdot)}}$,
which implies $v^{-1/m} \in {A_{(mp(\cdot))'}}$.
Since $p(\cdot)$ satisfies (\ref{equ.log1})
and (\ref{equ.log2}) then $mp(\cdot)$ and $(mp(\cdot))'$
also satisfy (\ref{equ.log1}) and (\ref{equ.log2}).
Note that $mp(\cdot)\in \mathscr{P}$, then, it follows from
Lemma \ref{lem.c-f-n} that
$M$ is bounded on $L^{(mp(\cdot))'}_{v^{-1/m}}(\rn)$.

So, to use Lemma \ref{lem.extra} for all ordered pairs in
$\mathcal{F}$, we need to check that
$\big\|T_k(\vec{f})\big\|_{L_v^{p(\cdot)}(\rn)}<\infty$
for every
$\vec{f}=(f_1,\cdots,f_m) \in \big(L_c^{\infty}(\rn)\big)^m$.
It is obviously always the case.
Indeed, since $v^{1/m} \in {A_{mp(\cdot)}}$ then
$v^{1/m} \in {L_{\mathrm{loc}}^{mp(\cdot)}(\rn)}$.
This together with Lemma \ref{lem.s-norm} gives
$v \in {L_{\mathrm{loc}}^{p(\cdot)}(\rn)}$, which implies
$\|v\chi_{B(0,k)}\|_{L^{p(\cdot)}(\rn)}<\infty$ for
each integer $k$. Thus
$$\big\|T_k(\vec{f})\big\|_{L_v^{p(\cdot)}(\rn)}
 =\big\|T_k(\vec{f})v\big\|_{L^{p(\cdot)}(\rn)}
 \le k\|v\chi_{B(0,k)}\|_{L^{p(\cdot)}(\rn)}<\infty.
$$

Now, we can apply Lemma \ref{lem.extra}
for $s=1/m$ to each pair in $\mathcal{F}$ and get
\begin{equation} \label{equ.proof1.3-1}
\big\|T_k(\vec{f})\big\|_{L_v^{p(\cdot)}(\rn)} \le
{C}\big\|{\cal{M}}_{L(\log{L})}(\vec{f})\big\|_{L_v^{p(\cdot)}(\rn)}.
\end{equation}

Recall the pointwise equivalence $M_{L(\log{L})}(g)(x)\approx
{M^2(g)(x)}$ for any locally integrable function $g$ (see (21) in
\cite{pe}), we have
$${\mathcal{M}}_{L(\log{L})}(\vec{f})(x)\le
\prod_{j=1}^{m}M_{L(\log{L})}(f_j)(x)
\le {C}\prod_{j=1}^{m} M^2(f_j)(x).
$$
Then, it follows from (\ref{equ.proof1.3-1}) and the
generalized H\"older's inequality (Lemma \ref{lem.gholder-m})
that
$$\big\|T_k(\vec{f})\big\|_{L_v^{p(\cdot)}(\rn)}
\le {C}\bigg\|\prod_{j=1}^{m}M^2(f_j) \bigg\|_{L_v^{p(\cdot)}(\rn)}
 \le {C}\prod_{j=1}^{m}\big\|M^2(f_j)\big\|_{L_{v_j}^{p_j(\cdot)}(\rn)}.
$$

Since $p_j(\cdot) \in \mathscr{P}$ and satisfies
(\ref{equ.log1}) and (\ref{equ.log2}),
and $v_j\in {A_{p_j(\cdot)}}$ for $j=1,\cdots,m$, then, by
applying Lemma \ref{lem.c-f-n} twice, we have
\begin{equation}  \label{equ.proof1.3-2}
\big\|T_k(\vec{f})\big\|_{L_v^{p(\cdot)}(\rn)}
 \le {C}\prod_{j=1}^{m} \|f_j\|_{L_{v_j}^{p_j(\cdot)}(\rn)}.
\end{equation}

Note that $\{T_k(\vec{f})(x)v(x)\}$ increases pointwise
almost everywhere to
$T_{\Pi\vec{b}}(\vec{f})(x)v(x)$ and $T_k(\vec{f})v \in {L^{p(\cdot)}(\rn)}$,
then,  Lemma \ref{lem.monotone} together with (\ref{equ.proof1.3-2}) gives
$$\|T_{\Pi\vec{b}}(\vec{f})\|_{L_v^{p(\cdot)}(\rn)}
 =\|T_{\Pi\vec{b}}(\vec{f})v\|_{L^{p(\cdot)}(\rn)}
 =\lim_{k\to\infty} \|T_k(\vec{f})v\|_{L^{p(\cdot)}(\rn)}
 \le {C}\prod_{j=1}^{m} \|f_j\|_{L_{v_j}^{p_j(\cdot)}(\rn)}.
$$

So complete the proof of Theorem \ref{thm.1.3}.
\end{proof}

\section{Applications}   \label{application}

In this section, we give some applications of the results obtained
above to paraproducts and bilinear pseudo-differential operators
with mild regularity.

\subsection{Paraproducts with mild regularity}

For $v\in {\mathbb{Z}}$ and $\kappa=(k_1,\cdots,k_n) \in
{\mathbb{Z}}^n$, let $P_{v\kappa}$ be the dyadic cube
$$P_{v\kappa}:=\big\{ (x_1,\cdots, x_n)\in\rn: k_i\le 2^{v}x_i
<k_i+1, ~ i=1,\cdots,n\big\}.
$$
The lower left-corner of $P:=P_{v\kappa}$ is $x_P=x_{v\kappa}:=
2^{-v\kappa}$ and the Lebesgue measure of $P$ is $|P|=2^{-vn}$. We
set
$${\cal{D}}=\big\{P_{v\kappa}: v\in {\mathbb{Z}},~ \kappa \in
{\mathbb{Z}}^n\big\}
$$
as the collection of all dyadic cubes.

\begin{definition}[\cite{mn}]
Let $\theta: [0,\infty)\to [0,\infty)$ be a nondecreasing and
concave function. An $\theta$-molecule associated to a dyadic cube
$P=P_{v\kappa}$ is a function $\phi_{P}=\phi_{v\kappa}: \rn \to
{\mathbb{C}}$ such that, for some $A_0>0$ and $N>n$, it satisfies
the decay condition
\begin{equation*}   
|\phi_P(x)|\le \frac{{A_0}2^{vn/2}}{(1+2^v|x-x_P|)^N},~~ x\in\rn,
\end{equation*}
and the mild regularity condition
\begin{equation*}   
\begin{split}
|\phi_{P}(x)-\phi_{P}(y)|\le {A_0}2^{vn/2}\theta(2^v|x-y|)
\left[\frac{1}{(1+2^v|x-x_P|)^N} +\frac{1}{(1+2^v|y-x_P|)^N}\right]
\end{split}
\end{equation*}
for all $x,y \in \rn$.
\end{definition}

\begin{definition}[\cite{mn}]
Given three families of $\theta$-molecules
$\{\phi_Q^j\}_{Q\in{\cal{D}}}$, $j=1,2,3$, the paraproduct
$\Pi(f,g)$ associated to these families is defined by
\begin{equation*}   
\Pi(f,g)=\sum_{Q\in{\cal{D}}} |Q|^{-1/2}
\big\langle{f,\phi_Q^1}\big\rangle
\big\langle{g,\phi_Q^2}\big\rangle \phi_Q^3, ~~~  f,g\in
{\mathscr{S}}(\rn).
\end{equation*}
\end{definition}

In \cite{mn}, some sufficient conditions on $\theta$ were given so
that the paraproducts defined above can be realized as bilinear
$\omega$-CZOs. The following result was proved in \cite[Theorem
5.3]{mn} when $\theta\in {Dini(1/2)}$. Indeed, the condition
$\theta\in {Dini(1/2)}$ can be reduced to $\theta\in {Dini(1)}$, see
Lemmas 8.2 and 8.3 in \cite{lz} for details.

\begin{lemma} \label{lem.mn.th5.3}
Let $\theta$ be concave and $\theta\in {Dini(1)}$, and let
$\{\phi_Q^j\}_{Q\in{\cal{D}}}$, $j=1,2,3$, be three families of
$\theta$-molecules with decay $N>10n$ and such that at least two of
them, say $j=1,2$, enjoy the following cancelation property
$$\int_{\rn}\phi_Q^j(x)dx=0, ~~ Q\in{\cal{D}}, ~ j=1,2,
$$
then $\Pi$ is a bilinear Calder\'on-Zygmund operator of type
$\omega$ with $\omega(t)=A_0^3A_{N}\theta(C_{N}t)$ and
$\tau=1/2$, where $A_{N}$ and $C_{N}$ are constants depending on
$N$.
\end{lemma}

Let ${\vec{b}}=(b_1, b_2) \in {BMO^2}$, the iterated commutator of
$\Pi$ with $\vec{b}$ is defined by
\begin{equation*}
\begin{split}
\Pi_{\Pi\vec{b}}(f_1,f_2)(x) &= [b_1,[b_2,T]_2,]_1(f_1,f_2)(x)\\
& =b_1(x)b_2(x)\Pi(f_1,f_2)(x)-b_2(x)\Pi(b_1f_1,f_2)(x) \\
&\qquad\quad -b_1(x)\Pi(f_1,b_2f_2)(x) + \Pi(b_1f_1,b_2f_2)(x).
\end{split}
\end{equation*}

Note that if $\theta\in {Dini(1)}$ and satisfies
\begin{equation}   \label{equ.ldini-2}
\int_0^1\frac{\theta(t)}{t}\bigg(1+\log\frac1t\bigg)^2 dt<\infty
\end{equation}
then $\omega(t)=A_0^3A_{N}\theta(C_{N}t)$ also belongs to $Dini(1)$
and satisfies (\ref{equ.ldini-2}). So, the following estimates for
iterated commutator $\Pi_{\Pi\vec{b}}$ are direct consequences of
Theorems \ref{thm.1.1}, \ref{thm.1.2} and \ref{thm.1.3} and Lemma
\ref{lem.mn.th5.3}.

\begin{theorem} \label{thm.cpara1}
Let $\theta$ and $\phi_Q^j$ be the same as in Lemma
\ref{lem.mn.th5.3}. Assume that $\theta$ satisfy (\ref{equ.ldini-2}).
If $\vec{b}\in {BMO^2}$ and ${\vec{w}}\in {A_{\vec{P}}}$ with
$1<p_1,p_2<\infty$ and $1/p=1/p_1+1/p_2$, then there exists a
constant $C>0$ such that
$$\big\|\Pi_{\Pi\vec{b}}(f_1,f_2)\big\|_{L^p(\nu_{\vec{w}})}\le
{C}\|b_1\|_{BMO}\|b_2\|_{BMO}
\|f_1\|_{L^{p_1}(w_1)}\|f_2\|_{L^{p_2}(w_2)}.
$$
\end{theorem}

\begin{theorem} \label{thm.cpara2}
Let $\theta$ and $\phi_Q^j$ be the same as in
Theorem \ref{thm.cpara1}. If $\vec{b}\in {BMO^2}$
and $\vec{w} \in {A_{(1,1)}}$, then there is a constant
$C>0$ depending on $\|\vec{b}\|_{BMO^2}$, such that
for all $\lambda>0$,
$$\nu_{\vec{w}}\left(\left\{ x\in \rn: \big|\Pi_{\Pi\vec{b}}(f_1,f_2)(x)
\big|>\lambda^2 \right\}\right) \le {C} \prod_{j=1}^2 \left(
\int_{\rn}\Phi^{(2)}\bigg(\frac{|f_j(x)|}{\lambda}\bigg)
w_j(x)dx\right)^{1/2}.
$$
\end{theorem}

\begin{theorem} \label{thm.cpara3}
Let $\theta$ and $\phi_Q^j$ be the same as in Theorem \ref{thm.cpara1}
and $\vec{b}\in {BMO^2}$. Suppose that $p(\cdot)\in \mathscr{P}_0$ and
$p_1(\cdot), p_2(\cdot) \in \mathscr{P}$ so that
${1}/{p(\cdot)} ={1}/{p_1(\cdot)}+ {1}/{p_2(\cdot)}$. If
$p(\cdot)$, $p_1(\cdot)$ and $p_2(\cdot)$ satisfy (\ref{equ.log1})
and (\ref{equ.log2}), then, for $v_i\in {A_{p_i(\cdot)}}$, $i=1,2$,
and $v=v_1v_2$, there exists a positive constant $C$ such that
$$\big\|\Pi_{\Pi\vec{b}}(f_1,f_2)\big\|_{L_v^{p(\cdot)}(\rn)}\le
{C} \|f_1\|_{L_{v_1}^{p_1(\cdot)}(\rn)} \|f_2\|_{L_{v_2}^{p_2(\cdot)}(\rn)}.
$$
\end{theorem}

\subsection{Bilinear pseudo-differential operators with mild regularity}

Let $m\in\mathbb{R}$, $0\le \delta,\rho\le 1$ and $\alpha, \beta,
\gamma \in {\mathbb{Z}}_{+}^n$. A bilinear pseudo-differential
operator $T_{\sigma}$ with a bilinear symbol $\sigma(x,\xi,\eta)$, a
priori defined from ${\mathscr{S}}(\rn) \times {\mathscr{S}}(\rn)$
to ${\mathscr{S'}}(\rn)$, is given by
$$T_{\sigma}(f_1,f_2)(x)=\int_{\rn}\int_{\rn}e^{ix\cdot(\xi+\eta)}
\sigma(x,\xi,\eta) \hat{f_1}(\xi) \hat{f_2}(\eta)d{\xi}d{\eta}.
$$
We say that a symbol $\sigma(x,\xi,\eta)$ belongs to the bilinear
H\"ormander class $BS^m_{\rho,\delta}$ if
$$\big|\partial_{x}^{\alpha}\partial_{\xi}^{\beta}
\partial_{\eta}^{\gamma} \sigma(x,\xi,\eta)\big| \le
{C_{\alpha,\beta}}
(1+|\xi|+|\eta|)^{m+\delta|\alpha|-\rho(|\beta|+|\gamma|)}, ~~
x,\xi,\eta\in\rn.
$$
for all multi-indices $\alpha, \beta$ and $\gamma$ and some constant
$C_{\alpha,\beta}$.

For $\Omega, \theta: [0,\infty) \to [0,\infty)$ and $0\le \rho\le
1$, we say that a symbol $\sigma \in {BS^m_{\rho,\theta,\Omega}}$ if
$$\big|\partial_{\xi}^{\alpha}\partial_{\eta}^{\beta}
\sigma(x,\xi,\eta)\big| \le {C_{\alpha,\beta}}
(1+|\xi|+|\eta|)^{m-\rho(|\alpha|+|\beta|)}
$$
and
\begin{align*}
&\big|\partial_{\xi}^{\alpha}\partial_{\eta}^{\beta}\big(
\sigma(x+h,\xi,\eta)-\sigma(x,\xi,\eta)\big)\big|\\
&~~~ \le {C_{\alpha,\beta}} \theta(|h|)\Omega(|\xi|+|\eta|)
(1+|\xi|+|\eta|)^{m-\rho(|\alpha|+|\beta|)}
\end{align*}
for all $x,\xi,\eta\in\rn$. Obviously, ${BS^{m}_{\rho,0}} \subset
{BS^m_{\rho,\theta,\Omega}}$.

The following result was proved by Maldonado and Naibo in
\cite[Theorem 4.3]{mn}.

\begin{lemma}  \label{lem.mn.th4.3}
Let $a\in (0,1)$, $\theta$ be concave with $\theta\in {Dini(a/2)}$
and $\Omega:[0,\infty) \to [0,\infty)$ be nondecreasing such that
\begin{equation}   \label{equ.pdo1} 
\sup_{0<t<1}\theta^{1-a}(t)\Omega(1/t)<\infty.
\end{equation}
If $\sigma \in {BS^0_{1,\theta,\Omega}}$ with $|\alpha|+|\beta|\le
4n+4$, then $T_{\sigma}$ is a bilinear Calder\'on-Zygmund operator
of type $\omega$ with $\omega(t)=\theta^a(t)$ and $\tau=1/3$.
\end{lemma}

Let ${\vec{b}}=(b_1, b_2) \in {BMO^2}$, the iterated commutator of
bilinear pseudo-differential operator $T_{\sigma}$ with
$\vec{b}=(b_1,b_2)$ is defined by
\begin{equation*}
\begin{split}
T_{\sigma,\Pi\vec{b}}(f_1,f_2)(x) &=
[b_1,[b_2,T_{\sigma}]_2,]_1(f_1,f_2)(x)\\
&=b_1(x)b_2(x)T_{\sigma}(f_1,f_2)(x)-b_2(x)
T_{\sigma}(b_1f_1,f_2)(x) \\
  &\qquad\quad -b_1(x)T_{\sigma}(f_1,b_2f_2)(x) +
  T_{\sigma}(b_1f_1,b_2f_2)(x).
\end{split}
\end{equation*}

For the iterated commutator of bilinear pseudo-differential
operators with associated symbols in ${BS^0_{1,\theta,\Omega}}$, we
have the following results.

\begin{theorem} \label{thm.cpdo1}
Let $a\in (0,1)$, $\theta$ be concave with $\theta\in {Dini(a/2)}$
and $\theta^a(t)$ satisfying (\ref{equ.ldini-2}), and
$\Omega:[0,\infty) \to [0,\infty)$ be nondecreasing such that
(\ref{equ.pdo1}) holds. Suppose that $\sigma \in
{BS^0_{1,\theta,\Omega}}$ with $|\alpha|+|\beta|\le 4n+4$. If
$\vec{b}\in {BMO^2}$ and ${\vec{v}}_w\in {A_{\vec{P}}}$ with
$1<p_1,p_2<\infty$ and $1/p=1/p_1+ 1/p_2$, then there exists a
constant $C>0$ such that
$$\big\|T_{\sigma,\Pi\vec{b}}(f_1,f_2)\big\|_{L^p(\nu_{\vec{w}})}\le
{C}\|\vec{b}\|_{BMO^2}\|f_1\|_{L^{p_1}(w_1)}\|f_2\|_{L^{p_2}(w_2)}.
$$
\end{theorem}

\begin{theorem}   \label{thm.cpdo2} 
Let $a, \theta$, $\Omega$ and $\sigma$ be the same as in Theorem
\ref{thm.cpdo1}. If $\vec{b}\in {BMO^2}$ and $\vec{w} \in
A_{(1,1)}$, then there is a constant $C>0$, depending on
$\|\vec{b}\|_{BMO^2}$, such that for any $\lambda>0$,
$$\nu_{\vec{w}}\left(\left\{ x\in \rn: \big|T_{\sigma,\Pi\vec{b}}
(f_1,f_2)(x) \big|>\lambda^2 \right\}\right) \le {C} \prod_{j=1}^2
\left( \int_{\rn}\Phi^{(2)}\bigg(\frac{|f_j(x)|}{\lambda}\bigg)
w_j(x)dx\right)^{1/2}.
$$
\end{theorem}

\begin{theorem} \label{thm.cpdo3}
Let $a, \theta$, $\Omega$ and $\sigma$ be the same as in Theorem
\ref{thm.cpdo1} and $\vec{b}\in {BMO^2}$.
Suppose that $p(\cdot)\in \mathscr{P}_0$ and
$p_1(\cdot), p_2(\cdot) \in \mathscr{P}$ so that
${1}/{p(\cdot)} ={1}/{p_1(\cdot)}+ {1}/{p_2(\cdot)}$.
If $p(\cdot)$, $p_1(\cdot)$ and $p_2(\cdot)$ satisfy (\ref{equ.log1})
and (\ref{equ.log2}), then, for $v_i\in {A_{p_i(\cdot)}}$, $i=1,2$,
and $v=v_1v_2$, there exists a constant $C>0$ such that
$$\big\|T_{\sigma,\Pi\vec{b}}(f_1,f_2)\big\|_{L_v^{p(\cdot)}(\rn)} \le
{C} \|f_1\|_{L_{v_1}^{p_1(\cdot)}(\rn)} \|f_2\|_{L_{v_2}^{p_2(\cdot)}(\rn)}.
$$
\end{theorem}

Since $\theta\in {Dini(a/2)}$ and $\theta^a(t)$ satisfies
(\ref{equ.ldini-2}) implies $\omega(t)=\theta^a(t) \in
{Dini(1/2)}\subset{Dini(1)}$ and $\omega$ satisfying
(\ref{equ.ldini-2}), then Lemma \ref{lem.mn.th4.3} together with
Theorems \ref{thm.1.1} -- \ref{thm.1.3} gives Theorems
\ref{thm.cpdo1} -- \ref{thm.cpdo3}, respectively.




\begin{thebibliography}{99}

\bibitem{cm1} R. R. Coifman and Y. Meyer, On commutators of singular
integrals and bilinear singular integrals, Trans. Amer. Math. Soc.,
{\bf 212} (1975), 315--331.

\bibitem{cm2} R. R. Coifman and Y. Meyer, Commutateurs
d'int\'egrales singuli\`eres et op\'erateurs multilin\'eaires, Ann.
Inst. Fourier (Grenoble), {\bf 28} (1978), no. 3, 177--202.

\bibitem{cu-f} D. Cruz-Uribe and A. Fiorenza, Variable Lebesgue
Spaces: Foundations and Harmonic Analysis, Birkh\"auser/Springer,
Basel (2013).


\bibitem{cfn} D. Cruz-Uribe, A. Fiorenza and C. J. Neugebauer,
Weighted norm inequalities for the maximal operator on variable
Lebesgue spaces, J. Math. Anal. Appl., {\bf 304} (2012), 744--760.

\bibitem{cu-w2014} D. Cruz-Uribe and L.-A. Wang, Variable Hardy
spaces, Indiana Univ. Math. J., {\bf 63} (2014), no.2, 447--493.

\bibitem{cu-w} D. Cruz-Uribe and L.-A. Wang, Extrapolation and weighted
norm inequalities in the variable Lebesgue spaces, to appear in Trans. Amer.
Math. Soc., arXiv 1408.4499v1.


\bibitem{dhhr} L. Diening, P. Harjulehto, P. H\"{a}st\"{o} and
M. R\r{u}\v{z}i\v{c}ka, Lebesgue and Sobolev Spaces with Variable
Exponents, Lecture Notes in Math. vol. 2017, Springer, Heidelberg
(2011).


\bibitem{dgy} X. T. Duong, L. Grafakos and L. Yan, Multilinear
operators with non-smooth kernels and commutators of singular
integrals, Trans. Amer. Math. Soc., {\bf 362} (2010), 2089--2113.

\bibitem{fs} C. Fefferman and E. M. Stein, $H^p$ spaces of several
variables, Acta Math., {\bf 129} (1972), 137--193.

\bibitem{gar-rub} J. Garc\'ia-Cuerva and J. L. Rubio de Francia,
Weighted Norm Inequalities and Related Topics, North-Holland Math.
Studies, vol. 116, North-Holland Publishing Co., Amsterdam (1985).

\bibitem{gra} L. Grafakos, Modern Fourier Analysis, 3rd ed., Grad.
Texts in Math., vol. 250, Springer, New York, 2014.

\bibitem{glmy} L. Grafakos, L. Liu, D. Maldonado and D. Yang,
Multilinear analysis on metric spaces, Dissertationes Math., {\bf
497} (2014), 121 pp.

\bibitem{gt1} L. Grafakos and R. H. Torres, Multilinear
Calder\'on-Zygmund theory, Adv. Math., {\bf 165} (2002), 124--164.

\bibitem{gt2} L. Grafakos and R. H. Torres, Maximal operator and
weighted norm inequalities for multilinear singular integrals,
Indiana Univ. Math. J., {\bf 51} (2002), 1261--1276.

\bibitem{h-x} A. W. Huang and J. S. Xu. Multilinear singular
integrals and commutators in variable exponent Lebesgue spaces,
Appl. Math. J. Chinese Univ. Ser. B, {\bf 25} (2010), no.1, 69--77.

\bibitem{ks} C. Kenig and E. M. Stein,  Multilinear estimates and
fractional integration. Math. Res. Lett., 6 (1999), no. 1, 1-15.

\bibitem{loptt} A. K. Lerner, S. Ombrosi, C. P\'erez, R. H. Torres
and R. Trujillo-Gonz\'alez, New maximal functions and multiple
weights for the multilinear Calder\'on-Zygmund theory, Adv. Math.,
{\bf 220} (2009), 1222--1264.

\bibitem{ls} K. Li and W. Sun, Weak and strong type weighted
estimates for multilinear Calder\'on-Zygmund operators, Adv. Math.,
{\bf 254} (2014), 736--771.

\bibitem{ll} Z. Liu and S. Lu, Endpoint estimates for commutators of
Calder\'on-Zygmund type operators, Kodai Math. J., {\bf 25} (2002),
no. 1, 79--88.

\bibitem{lz} G. Lu and P. Zhang, Multilinear Calder\'on-Zygmund
operators with kernels of Dini's type and applications, Nonlinear
Analysis, {\bf 107} (2014), 92--117.

\bibitem{mn} D. Maldonado and V. Naibo, Weighted norm inequalities
for paraproducts and bilinear pseudodifferential operators with mild
regularity, J. Fourier Anal. Appl., {\bf 15} (2009), 218--261.

\bibitem{pe} C. P\'erez, Endpoint estimates for commutators of
singular operators, J. Funct. Anal., {\bf 128} (1995), 163--185.

\bibitem{pptt} C. P\'erez, G. Pradolini, R. H. Torres and R.
Trujillo-Gonz\'alez, End-point estimates for iterated commutators of
multilinear singular integrals, Bull. London Math. Soc., {\bf 46}
(2014), 26--42.

\bibitem{pt2014} C. P\'erez and R. H. Torres, Minimal regularity
conditions for the end-point estimate of bilinear Calder\'on-Zygmund
operators, Proc. Amer. Math. Soc. Series B, {\bf 1} (2014), 1--13.

\bibitem{rr} M. M. Rao and Z. D. Ren, Theory of Orlicz Spaces,
Marcel Dekker Inc., New York (1991).

\bibitem{s} E. M. Stein, Harmonic Analysis: Real Variable Methods,
Orthogonality, and Oscillatory Integrals, Princeton Univ. Press,
Princeton, New Jersey (1993).

\bibitem{y} K. Yabuta, Generalizations of Calder\'on-Zygmund
operators, Studia Math., {\bf 82} (1985), no. 1, 17--31.

\bibitem{zx} P. Zhang and H. Xu, Sharp weighted estimates for
commutators of Calder\'on-Zygmund type operators, Acta Math. Sinica
(Chinese Series), {\bf 48} (2005), no. 4, 625--636.


\end{thebibliography}
\end{document}